\documentclass
[
    a4paper,
    DIV=11,
]
{scrartcl}
\usepackage{fullpage,authblk,amsmath,amssymb,amsthm,amsfonts}
\usepackage{bm, bbm, dsfont}
\usepackage{enumerate}
\usepackage{tikz}
\usepackage{caption}
\usepackage{subcaption}
\usepackage{url,mathtools}
\usepackage{tensor}
\usepackage{footmisc}
\usepackage{graphicx}
\usepackage{mathrsfs}
\usepackage{thm-restate}

\usepackage[pdffitwindow=true,
		   pdfstartview={FitH},
            plainpages=false,
            pdfpagelabels=true,
            pdfpagemode=UseOutlines,
            pdfpagelayout=SinglePage,
            bookmarks=false,
            colorlinks=true,
            hyperfootnotes=false,
            linkcolor=blue,
            urlcolor=blue!30!black,
            citecolor=green!50!black]{hyperref}

\usepackage{cleveref}

\usepackage{comment} 

\newtheorem{theorem}{Theorem}[section]

\newtheorem{lemma}[theorem]{Lemma}
\newtheorem{proposition}[theorem]{Proposition}

\newtheorem{definition}[theorem]{Definition}

\theoremstyle{definition}
\newtheorem{assumption}[theorem]{Assumption}
\newtheorem{remark}[theorem]{Remark}

\newcommand{\R}[0]{\mathbb{R}}
\newcommand{\N}[0]{\mathbb{N}}

\newcommand{\mfrak}[1]{\mathfrak{#1}}
\newcommand{\mscr}[1]{\mathscr{#1}}
\newcommand{\mcal}[1]{\mathcal{#1}}
\newcommand{\bb}[1]{\mathbb{#1}}
\newcommand{\norm}[1]{\Vert #1\Vert}
\newcommand{\Norm}[1]{\left\Vert #1 \right\Vert}
\newcommand{\rd}[0]{\mathrm{d} }
\newcommand{\abs}[1]{\vert #1 \vert}
\newcommand{\ang}[1]{\langle #1 \rangle}
\newcommand{\Abs}[1]{\left\vert#1\right\vert}
\newcommand{\Ang}[1]{\left\langle #1 \right\rangle}

\newcommand{\closure}[1]{\overline{#1}} 

\newcommand{\covariate}[0]{x} 
\newcommand{\covariatespace}[0]{\mcal{X}} 
\newcommand{\covariatelaw}[0]{\law{X}} 

\newcommand{\evalmap}[0]{\mscr{E}} 

\newcommand{\paramspace}[0]{\mcal{P}} 

\newcommand{\cpm}[0]{\theta}  
\newcommand{\cpmspace}[0]{\Theta}  

\newcommand{\regspace}[0]{\mathcal{R}} 

\newcommand{\Lipschitz}[0]{\textup{Lip}}

\newcommand{\fwdmodel}[0]{\mcal{G}} 

\newcommand{\noise}[0]{\varepsilon} 

\newcommand{\informoperator}[0]{\bb{I}_{\cpm_0}^{\ast}\bb{I}_{\cpm_0}} 

\newcommand{\normaldist}[2]{\mcal{N}\left(#1,#2\right)} 
\newcommand{\distiid}[0]{\stackrel{\textup{iid}}{\sim}}

\newcommand{\spn}[0]{\textup{span}} 
\newcommand{\law}[1]{\mu_{#1}} 

\newcommand*{\todo}[1]{\bgroup\color{red}TODO: #1\egroup}

\DeclareMathOperator*{\dimension}{dim}

\numberwithin{equation}{section}

\begin{document}
\title{Bayesian inference of covariate-parameter relationships for population modelling}

\author{Han Cheng Lie\thanks{~han.lie@uni-potsdam.de}}
\affil{~Institut f\"ur Mathematik, Universit\"at Potsdam, Campus Golm, Potsdam OT Golm 14476, Germany}
\renewcommand\Affilfont{\small}

\date{}
\maketitle

\begin{abstract}
We consider population modelling using parametrised ordinary differential equation initial value problems (ODE-IVPs). For each individual drawn randomly from the unknown population distribution, the corresponding parameters for the ODE-IVP cannot be measured directly, but a vector of covariates is given, and one component of the solution to the corresponding ODE-IVP is observed at a fixed finite time grid.  
The task is to identify a covariate-parameter relationship that maps covariate vectors to parameter vectors.
Such settings and problems arise in pharmacokinetics, where the observations are blood drug concentrations, the covariates are clinically observable quantities, and the covariate-parameter relationship is used for personalised drug dosing.
For linear homogeneous ODE-IVPs with vector fields defined by matrices that are diagonalisable over $\mathbb{R}$, and for fixed time and random covariate design, we use recent results of Nickl et al. for Bayesian nonlinear statistical inverse problems, to prove posterior contraction and Bernstein--von Mises results for the unknown covariate-parameter relationship. We analytically demonstrate our results on an example from the pharmacokinetics literature.
\end{abstract}



\section{Introduction}
\label{sec_introduction}

Many mathematical models for deterministic continuous dynamical systems are expressed in terms of ordinary differential equation initial value problems (ODE-IVPs) where the vector field or the initial condition depend on a vector of parameters.
An important task when applying such parametrised ODE-IVPs is to identify suitable parameter values so that the models better fit experimental data, or have better predictive power.
This task becomes more challenging when the task is to model individuals in a population using the same ODE but different parameters, and when the number of observations of the ODE solution for each individual is limited, so that estimates of the parameter vector for each individual may not be available or sufficiently accurate.
Such population modelling and statistical estimation problems arise frequently in pharmacokinetics, in the context of compartment models; see e.g. \cite{Wakefield_etal_1999,Joerger2012}.

A two-compartment model involves representing a patient as a two-dimensional dynamical system $s(t)=(s_1(t),s_2(t))$, where $s_1(t)$ describes the concentration of a drug in the so-called `central compartment' of the patient and $s_2(t)$ describes the concentration of the drug in the so-called `peripheral compartment'. The `volume' of the central and peripheral compartments is denoted by $V_1$ and $V_2$ respectively. In many cases, the central compartment includes the circulatory system together with the target region of the body, and the peripheral compartment is the complement of the central compartment.
Suppose that the rate of flow between compartments is equal to $Q$ in both directions, that the drug dosage is proportional to the patient's weight $w$ and administered directly to the central compartment at time $0$, and that the drug is removed from the patient's body via the same compartment at some `elimination clearance' rate $CL$. This yields the following ODE-IVP on a predetermined time interval $[0,T]$:
\begin{equation}
\label{eq_ODE_IVP_two_compartment_model}
 \begin{aligned}
    V_1 \frac{\rd }{\rd t} s_1 (t)=&Q(s_2(t)-s_1(t))-CL\ s_1(t),& V_1 s_1(0)&=D_0 w,
    \\ 
    V_2\frac{\rd}{\rd t} s_2(t)=&Q(s_1(t)-s_2(t)),& s_2(0)&=0.
\end{aligned}
\end{equation}
For the ODE-IVP above, the vector $(CL,V_1,Q,V_2)$ of parameters cannot be observed or measured in a clinical setting.
The task of estimating the vector $(CL,V_1,Q,V_2)$ of parameters for each patient using observations of the concentration $s_1(t)$ at discrete times $t_1,\ldots, t_{d_o}$ is often complicated by the fact that only a small number of blood samples per patient may be collected; see e.g. \cite[Section 3]{Wakefield_etal_1999} and \cite[p. 136]{Mesnil1998}.

One approach to tackling these challenges is to exploit the information that is available in the form of so-called `covariates', such as the gender, age, or weight of the patient.
Unlike the parameters of the ODE-IVP, the covariates can be measured in clinical settings.
Given a choice of covariates and a parametrised ODE-IVP, one needs to find a function that maps the vector of covariates to the vector of parameters.
In the pharmacokinetics literature, this function is sometimes referred to as a `parameter-covariate relationship' or a `covariate-parameter relationship'. 
We shall refer to this function as the `covariate-to-parameter map' (CPM).

The importance of CPMs in pharmacology is that they can be used to determine dosing regimes that are tailored to each patient: given a patient's covariate vector, the CPM indicates the corresponding parameter vector that can be used in the ODE-IVP to model how the drug concentration in the patient evolves over time.
Tailored dosing regimes can help to improve the efficacy of the drug treatment and reduce the risk of adverse side-effects.

In many cases, the CPM is chosen to be a simple function of a covariate. 
Piecewise linear functions, power functions, and exponential functions are frequently used; see e.g. \cite[p. 1465]{Jonsson1998}, \cite[pp. 121-122]{Joerger2012}, and \cite[Equation (2.11)]{Tan2019}.
The choice of functional forms of the covariate-parameter relationship is often motivated by convenience \cite[p. 51]{Lai2006}.
This raises the question of whether one can infer the true CPM in a more data-driven way, under fewer assumptions on the functional form of the CPM.

The goal of this paper is to establish some mathematical foundations for a Bayesian statistical approach for inferring the CPM from partial observations of solutions to parametrised linear homogeneous ODE-IVPs, where the matrix that defines the vector field is diagonalisable over $\R$, and only finitely many observations of one component of the solution are available. We demonstrate the applicability of this approach on an example from pharmacokinetics.

\paragraph{Outline} 
In \Cref{sec_properties_of_solutions_to_linear_ODE_IVPs}, we consider the collection of parametrised linear homogeneous ODE-IVPs with diagonalisable right-hand sides. 
We state and prove some properties of the function that maps a parameter vector to the corresponding vector of observations of one component of the solution. 
In \Cref{sec_forward_operator_fixed_time_random_covariate_design}, we use the results from \Cref{sec_properties_of_solutions_to_linear_ODE_IVPs} to analyse forward operators associated to a fixed, finite time design and random covariate design.
The main results in this section consist of local boundedness, local Lipschitz continuity, and stability estimates of the forward operator.
In \Cref{sec_properties_of_Bayesian_inverse_problem}, we define a Bayesian inverse problem in which the object to be inferred is the true CPM that is valid for the entire population.
We apply the framework developed by Nickl et al. that is presented in \cite{Nickl2023}, and define a family of posteriors associated to a family of rescaled Gaussian priors.
For these posteriors, we establish posterior contraction, local asymptotic normality, and a Bernstein--von Mises theorem.
In \Cref{sec_application_two_compartment_model}, we apply a posterior contraction result from \Cref{sec_properties_of_Bayesian_inverse_problem} to an example of a two-compartment model from the pharmacokinetics literature \cite{Hartung2021,Robbie2012}.
We conclude in \Cref{sec_conclusion}. 

\paragraph{Contributions} Our analysis uses the theoretical framework presented in \cite{Nickl2023} for Bayesian nonlinear statistical inverse problems with rescaled Gaussian process priors.
In this framework, the key challenge lies in verifying that the forward operator satisfies certain regularity and stability estimates.
We are not aware of any other work in the literature where this framework is applied to infer CPMs in parametrised ODE-IVPs.
We believe that the estimates in \Cref{sec_forward_operator_fixed_time_random_covariate_design} and \Cref{sec_properties_of_Bayesian_inverse_problem} are new, both in the context of parametrised ODE-IVPs and in the context of statistical inference for pharmacokinetics.
In addition, while we expect that some of the results in \Cref{sec_properties_of_solutions_to_linear_ODE_IVPs} are known in the theory of ODEs, \Cref{lem_locally_Lipschitz_inverse_evaluation_map} --- which states the local Lipschitz continuity of the inverse of the parameter to observable map --- may be new.

Our strategy is to prove estimates on the forward operator that hold uniformly with respect to the covariate. 
In this way, we can prove estimates of the form given in \cite[Chapter 2]{Nickl2023} without imposing any regularity assumptions on the unknown covariate distribution, such as admitting a Lebesgue density.
This is advantageous for applications in pharmacokinetics, where in general one has very little knowledge about the true population distribution of covariates, and where the covariates may take discrete values.

\subsection{Related work}
\label{ssec_related_work}

For deterministic differential equations, the work \cite{Chkrebtii2016} considers the approximation of solutions of deterministic differential equations by numerical integration methods, and shows the consistency of the family of posteriors indexed by the resolution parameter of the numerical method.
In contrast, we consider standard posterior contraction in the `classical' limit of infinite data.

The papers \cite{Bhaumik2015,Bhaumik2017,Tan2019} develop a two-step approach to the task of parameter inference for ODEs, where the data are obtained by discrete observations of the solutions to the parametrised ODE and the ODE does not admit a closed-form solution, but the ODE vector field is known.
In the first step, the parameter inference problem is expressed as a problem of regressing the data against the unknown parameters, where the regression function is modelled nonparametrically, e.g. using B-splines.
The second step uses the `derivative matching' idea in that a minimisation problem is solved to find the parameter values that best match the observed values of the ODE vector field.
The recent paper \cite{Bhaumik2022} develops this approach to the case where the constraint is expressed by a PDE instead of an ODE.
For the above-mentioned papers, Bayesian methods are used to establish properties such as posterior consistency and Bernstein--von Mises theorems in the random design setting, for a finite-dimensional vector of unknowns.
In contrast, we aim to infer an infinite-dimensional object, namely the CPM that maps every covariate vector to the corresponding vector of parameters to be used in the ODE-IVP, without considering a specific basis.

The review \cite{Joerger2012} presents some ideas of covariate pharmacokinetic model building, and \cite{Wakefield_etal_1999} focuses on the Bayesian approach to population pharmacokinetics.
The work \cite{Maitre1991} uses Bayesian regression to estimate pharmacokinetic parameters from observed drug concentrations.
In \cite{Mandema1992}, generalised additive models and splines were used in a parametric approach to describe the CPM.
This approach was further analysed and modified in \cite{Jonsson1998}.
In contrast to the parametric or finite-dimensional approaches described above, the paper \cite{Mesnil1998} used the nonparametric maximum likelihood estimation to estimate the joint distribution of parameters and covariates, but not the CPM itself. 
The work \cite{Lai2006} uses regression splines or neural networks to estimate the CPM, but does not prove posterior contraction.
\cite{Hartung2021} tests the goodness-of-fit of certain parametric classes of covariate-parameter relationships against nonparametric alternatives, using kernel-based Tikhonov regularisation and ideas from statistical learning, but does not consider a Bayesian approach.
The recent work \cite{Abhishake2023} studies the problem of predicting changes in drug concentrations for pharmacokinetic models, but uses regularisation schemes and statistical learning instead of a Bayesian approach.

\subsection{Notation}
\label{ssec_notation}

For $n\in\N$, $[n]\coloneqq \{1,\ldots,n\}$. For $x\in\R^{n}$, $A\in\R^{m\times n}$, and $0<q\leq \infty$, $\norm{x}_{q}\coloneqq (\sum_{i=1}^{n}\abs{x_i}^{q})^{1/q}$, $\norm{A}_{q}\coloneqq  \sup_{\norm{x}_{q}\leq 1}\norm{Ax}_{q}$, and $B_2(x,r)\coloneqq \{z: \norm{z-x}_2<r\}$.
We denote the set of invertible elements of $\R^{n\times n}$ by $GL(n,\R)$.
For $a,b\in\R$, $a\wedge b\coloneqq \min\{a,b\}$ and $a\vee b\coloneqq \max\{a,b\}$.
We denote the closure and cardinality of a subset $A$ of some Euclidean space by $\closure{A}$ and $\# A$ respectively. 

If $(V,\norm{\cdot}_V)$ is a Banach space and $W\subset V$ is a linear subspace, then we denote by $\overline{W}\vert_{V}$ the closure of $W$ with respect to the norm $\norm{\cdot}_V$, and $B_{V}(x,r)\coloneqq \{v\in V\ :\ \norm{v-x}_V<r\}$.
For $d_1,d_2\in\N$, nonempty Borel sets $D_1\subseteq\R^{d_1}$,  $D_2\subseteq\R^{d_2}$, a measure $\mu$ on $D_1$, and $0<q\leq\infty$, let $\norm{f}_{L^q_\mu}\coloneqq(\int_{D_1} \abs{f(x)}_{2}^{q}\mu(\rd x))^{1/q}$ and $L^q_\mu(D_1,D_2)$ denote the corresponding Banach space. Let $\norm{f}_\infty\coloneqq \sup_{x\in D_1}\norm{f(x)}$.
For $\beta\geq 0$, denote by $H^{\beta}_{\mu}(D_1,D_2)$ the Sobolev space of functions from $D_1$ to $D_2$ with weak derivatives of up to order $\beta$ that are $L^2_{\mu}$-integrable, with $H^0_{\mu}(D_1,D_2)=L^2_{\mu}(D_1,D_2)$.
Denote the space of bounded, continuous functions from $D_1$ to $D_2$ by $C_b(D_1,D_2)$.

The notation $a\leftarrow b$ means that $a$ is replaced with $b$, while $a\lesssim b$ means that $a\leq Cb$ for some positive scalar $C$ that does not depend on $a$ or $b$.

We fix a common underlying probability space. Given a random variable $Z$ defined on this probability space, $\law{Z}$ denotes the law of $Z$.

We denote the set of admissible covariates by $\covariatespace\subset \R^{d_\covariate}$, the set of probability measures on $\covariatespace$ by $\mcal{M}_1(\covariatespace)$, the true unknown population distribution of covariates by $\covariatelaw$, the open set of admissible parameters by $\paramspace\subset \R^{d_p}$, and the set of admissible CPMs by $\cpmspace$.
Both $\covariatespace$ and $\paramspace$ have nonempty interior.
We denote an arbitrary element of $\covariatespace$, $\paramspace$, and $\cpmspace$ by $x$, $p$, and $\cpm$ respectively.
In \Cref{sec_forward_operator_fixed_time_random_covariate_design}, we use $\cpm_0$ to denote a reference CPM, and in \Cref{sec_properties_of_Bayesian_inverse_problem} the reference CPM $\cpm_0$ is taken to be the true data-generating CPM.

\section{Properties of solutions to linear ODE-IVPs}
\label{sec_properties_of_solutions_to_linear_ODE_IVPs}

In this section, we specify the type of ODE-IVPs that we shall consider and state some results about their solutions.
These results are essential for proving desirable properties of the forward operator --- e.g. local Lipschitz continuity and stability --- that we introduce in \Cref{sec_forward_operator_fixed_time_random_covariate_design}. This forward operator defines the Bayesian nonlinear statistical inverse problem that we analyse in \Cref{sec_properties_of_Bayesian_inverse_problem}.

Fix a bounded time interval $[0,T]$ for some $0<T<\infty$. 
Consider a time-homogeneous, linear ODE-IVP in $\R^{d_s}$, $d_s\in\N$, where both the matrix $A:\R^{d_p}\to\R^{d_s\times d_s}$ and the initial condition $s_0:\R^{d_p}\to\R^{d_s}$ depend on a parameter vector $p\in\R^{d_p}$:
\begin{equation}
 \label{eq_ODE_IVP}
s(0,p)=s_0(p)\in\R^{d_s},\quad \frac{\rd}{\rd t} s(t,p)=A(p)s(t,p),\quad t\in[0,T].
\end{equation}
We refer to $(s(t,p))_{t\in[0,T]}$ as the `solution' of \eqref{eq_ODE_IVP} for the parameter $p$, and $s(t,p)$ as the `state' at time $t$ of this solution.

The following assumption states that both the matrix $A(p)$ and initial condition $s_0(p)$ in \eqref{eq_ODE_IVP} are locally bounded and locally Lipschitz continuous functions of $p$.
\begin{assumption}
    \label{asmp_local_boundedness_and_local_lipschitz_continuity_of_matrix_and_IC}
   For every $M>0$ there exists some $C_1(M)>0$ such that 
\begin{subequations}
\begin{align}
\sup\{ \norm{A(p)}_2\vee \norm{s_0(p)}_2:p\in B_2(0,M)\}\leq& C_1(M),    
\label{eq_local_bound_on_operator_norm_of_A_and_IC}
\\
\sup\left\{ \norm{A(p)-A(q)}_2\vee \norm{s_0(p)-s_0(q)}_2:p,q\in B_2(0,M)\right\}\leq &C_1(M)\norm{p-q}_2.
\label{eq_local_bound_on_Lipschitz_constant_of_A_and_IC}
\end{align}
\end{subequations}
\end{assumption}
The next result shows that \Cref{asmp_local_boundedness_and_local_lipschitz_continuity_of_matrix_and_IC} implies local boundedness and local Lipschitz continuity of the solution of \eqref{eq_ODE_IVP}, viewed as a function of the parameter $p$.
We expect that this result is known in the theory of parametrised ODE-IVPs, but state and prove it for the sake of completeness.
\begin{lemma}
\label{lem_local_boundedness_of_solution_operator}
Suppose \Cref{asmp_local_boundedness_and_local_lipschitz_continuity_of_matrix_and_IC} holds with constant $C_1(M)$ for every $M$.
Then for every $M>0$ and $t\in [0,T]$, 
\begin{equation}
    \label{eq_local_boundedness_of_solution_operator}
    \forall (t,p)\in [0,T]\times B_2(0,M),\quad \norm{s(t,p)}_2\leq e^{C_1(M)t}C_1(M),
\end{equation}
and there exists $L=L(C_1(M),T)>0$ such that
\begin{equation}
\label{eq_local_Lipschitz_continuity_of_solution_operator}
\forall p,q\in B_2(0,M),\quad \sup_{t\in[0,T]}\norm{s(t,p)-s(t,q)}_2\leq L\norm{p-q}_2.
\end{equation}
\end{lemma}
\begin{proof}[Proof of \Cref{lem_local_boundedness_of_solution_operator}]
Let $t\in[0,T]$ be arbitrary.
Since $\norm{Ax}_q\leq \norm{A}_q\norm{x}_q$ for any $A\in\R^{n\times n}$ and $x\in\R^n$, we obtain 
 \begin{equation*}
  \norm{s(t,p)}_2=\Norm{e^{A(p)t}s_0(p)}_2\leq \Norm{e^{A(p)t}}_2\norm{s_0(p)}_2\leq e^{\norm{A(p)}_2t} \norm{s_0(p)}_2\leq e^{C_1(M)t} C_1(M),
 \end{equation*}
 where the final inequality follows from \eqref{eq_local_bound_on_operator_norm_of_A_and_IC}. This proves \eqref{eq_local_boundedness_of_solution_operator}.
 For the second statement, 
 \begin{align*}
  &\norm{s(t,p)-s(t,q)}_2
  \\
  =&\Norm{s_0(p)+\int_0^t A(p)s(r,p)\rd r-s_0(q)-\int_0^t A(q)s(r,q)\rd r}_2
  \\
  \leq &\norm{s_0(p)-s_0(q)}_2+\Norm{\int_0^t A(p)s(r,p) -A(p)s(r,q) +A(p)s(r,q) - A(q)s(r,q)\rd r }_2
  \\
  \leq &\norm{s_0(p)-s_0(q)}_2+\Norm{\int_0^t A(p)\left(s(r,p) -s(r,q)\right)\rd r}_2+\Norm{\int_0^t \left(A(p)-A(q)\right)s(r,q)\rd r }_2
  \\
  \leq & \norm{s_0(p)-s_0(q)}_2+\norm{A(p)}_2 \int_0^t \Norm{s(r,p) -s(r,q)}_2\rd r +\Norm{A(p)-A(q)}_2\int_0^t \norm{s(r,q)}_2~\rd r 
  \\
  \leq & C_1(M)\norm{p-q}_2+C_1(M) \int_0^t \Norm{s(r,p) -s(r,q)}_2\rd r +C_1(M)\norm{p-q}_2~ e^{C_1(M)t}C_1(M)t
  \\
  \leq & C_1(M)\left(1+e^{C_1(M)t}C_1(M)t\right) \norm{p-q}_2+C_1(M) \int_0^t \Norm{s(r,p) -s(r,q)}_2\rd r,
 \end{align*}
 where the final two inequalities use \eqref{eq_local_bound_on_operator_norm_of_A_and_IC} and \eqref{eq_local_bound_on_Lipschitz_constant_of_A_and_IC}.
Gronwall's inequality yields \eqref{eq_local_Lipschitz_continuity_of_solution_operator} with $L(C_1(M),T)\coloneqq  e^{C_1(M)T}C_1(M)(1+e^{C_1(M)T}C_1(M)T)$.
\end{proof}

Next, we assume that the ODE matrix $A(p)$ is diagonalisable over $\R$.
\begin{assumption}[Diagonalisable ODE matrix]
 \label{asmp_ODE_matrix_diagonalisable}
 For every $p\in\paramspace$, there exists $\Lambda(p),V(p)\in GL(d_s,\R)$ with diagonal $\Lambda(p)$, such that $A(p)$ in \eqref{eq_ODE_IVP} satisfies $A(p)=V(p)\Lambda(p) V^{-1}(p)$.
\end{assumption}
We show in \Cref{sec_application_two_compartment_model} that there exists a two-compartment model from pharmacokinetics that satisfies  \Cref{asmp_ODE_matrix_diagonalisable}.

Given \eqref{eq_ODE_IVP}, it is known that $s(t,p)=e^{A(p)t}s_0(p)$. \Cref{asmp_ODE_matrix_diagonalisable} is important for our analysis, because it implies that $s(t,p)=V(p)\exp( \Lambda(p) t) V(p)^{-1}s_0(p)$. In particular, every component $s_i$, $i\in[d_s]$, of the solution of \eqref{eq_ODE_IVP} is a linear combination of exponential functions.
We fix an arbitrary choice and consider the first component of the solution from now on. 
For every $p\in\paramspace$ there exist $(\widehat{a}_i(p))_{i\in[d_s]},(\widehat{\lambda}_i(p))_{i\in[d_s]}\in\R^{d_s}$ such that
\begin{equation*}
 \forall t\geq 0,\quad s_1(t,p)=\sum_{i\in[d_s]} \widehat{a}_i(p)e^{\widehat{\lambda}_i(p) t},
\end{equation*}
where $\widehat{a}(p)\in\R^{d_s}$ depends only on the eigenvector matrix $V(p)$ and initial condition $s_0(p)$ and satisfies $\sum_{i\in[d_s]} \widehat{a}_i(p)=s_1(0,p)$, and the $(\widehat{\lambda}_i(p))_{i\in[d_s]}$ are the diagonal entries of $\Lambda(p)$, i.e. the eigenvalues of $A(p)$.

Now consider the functions $(\widehat{\lambda}_i)_{i\in[d_s]}$.
If there exist distinct $k,\ell\in[d_s]$ with $\widehat{\lambda}_k=\widehat{\lambda}_\ell$ on $\paramspace$, then $\widehat{a}_k(p)e^{\widehat{\lambda}_k(p)t}+\widehat{a}_\ell(p)e^{\widehat{\lambda}_\ell(p)t}=(\widehat{a}_k(p)+\widehat{a}_\ell(p))e^{\widehat{\lambda}_k(p)t}$ for every $t$.
Thus, if $\mfrak{d}=\mfrak{d}(p)$ denotes the number of distinct eigenvalues of $\Lambda(p)$, then we can rewrite the linear combination of $d_s$ exponential functions as a linear combination of $\mfrak{d}$ many exponential functions, where each exponential function is defined by two scalars, i.e. the prefactor and the eigenvalue.
This motivates the following definition.
\begin{definition}
 \label{def_intrinsic_dimension_of_ODE_IVP_solution_first_component}
 Suppose \Cref{asmp_ODE_matrix_diagonalisable} holds, and let $\{\lambda_j: j\in[\mfrak{d}]\}$ be the distinct elements of the set $\{\widehat{\lambda}_i:i\in[d_s]\}$. Then the \emph{coefficient map} of $t\mapsto s_1(t,p)$ is the map $p\mapsto (a(p),\lambda(p))\in\R^{2\mfrak{d}}$, where
\begin{equation}
   \label{eq_ODE_IVP_solution_first_component_exponential_form}
 \forall t\geq 0,\quad s_1(t,p)=\sum_{i\in[\mfrak{d}]} a_i(p)e^{\lambda_i(p) t},
\end{equation}
and the \emph{intrinsic dimension} of $t\mapsto s_1(t,p)$ is $2\mfrak{d}$.
\end{definition}
We will assume that the Jacobian of the coefficient map has full rank.
\begin{assumption}
 \label{asmp_coefficient_map_Jacobian_full_rank}
 The coefficient map $\paramspace\ni p\mapsto (a(p),\lambda(p))\in (\R\setminus\{0\})^{\mfrak{d}}\times \R^{\mfrak{d}}$ is $C^1$, and for every $q\in\paramspace$,
 \begin{equation}
  \label{eq_Jacobian_coefficient_map}
  \mcal{J}(q)\coloneqq \begin{bmatrix}
                        \nabla_p a_1(p)\vert \cdots \vert \nabla_p a_{\mfrak{d}}(p)\vert \nabla_p \lambda_1 (p)\vert \cdots\vert \nabla_p \lambda_{\mfrak{d}}(p)
                       \end{bmatrix}^\top \bigr\vert_{p=q} \in\R^{2 \mfrak{d} \times d_p}
 \end{equation}
has full rank. 
\end{assumption}
A necessary condition for \Cref{asmp_coefficient_map_Jacobian_full_rank} to hold is that the $(a_i)_{i\in[\mfrak{d}]}$ are distinct: if there exist $i,j\in[\mfrak{d}]$ such that $a_i=a_j$, then the corresponding two columns of $\mcal{J}(q)$ will be identical and thus $\mcal{J}(q)$ cannot be full rank.
This further motivates the definition of $2\mfrak{d}$ as the intrinsic dimension of $t\mapsto s_1(t)$.
We will use the requirement that $a(p)\in(\R\setminus\{0\})^{\mfrak{d}}$ in the proof of \Cref{lem_locally_Lipschitz_inverse_evaluation_map} below.

Fix an arbitrary finite time design, i.e. an arbitrary collection $(t_j)_{j\in [d_o]}\subset [0,T]$ of $d_o\in\N$ distinct observation times.
Define the map that evaluates $s_1$ at $(t_j)_{j\in[d_o]}$:
\begin{equation}
\label{eq_evaluation_map}
 \R^{d_p}\supseteq \paramspace\ni p\mapsto \evalmap(p)\equiv \evalmap(p;(t_j)_{j\in[d_o]})\coloneqq (s_1(t_j,p))_{j\in[d_o]}\in\R^{d_o}.
\end{equation}
The result below uses \Cref{asmp_ODE_matrix_diagonalisable} and \Cref{asmp_coefficient_map_Jacobian_full_rank} to provide a sufficient condition for the restriction of $\evalmap$ to an arbitrary ball to have an inverse, and for this inverse to be Lipschitz.
By \Cref{asmp_ODE_matrix_diagonalisable}, we can use \eqref{eq_ODE_IVP_solution_first_component_exponential_form}, i.e. the fact that the observed component $s_1$ of the solution of the ODE-IVP \eqref{eq_ODE_IVP} is a sum of exponential functions. By \Cref{asmp_coefficient_map_Jacobian_full_rank}, we obtain a lower bound on the size $d_o$ of the fixed, finite time design $(t_j)_{j\in[d_o]}$.
We use the result below to prove a so-called `stability estimate' on the forward operator that we introduce in \Cref{sec_forward_operator_fixed_time_random_covariate_design}.
In the framework presented in \cite{Nickl2023}, stability estimates are key to obtaining posterior contraction results.
\begin{lemma}[Locally Lipschitz continuity of inverse evaluation map]
 \label{lem_locally_Lipschitz_inverse_evaluation_map}
 Suppose that \Cref{asmp_ODE_matrix_diagonalisable} and \Cref{asmp_coefficient_map_Jacobian_full_rank} hold. 
 If $d_o\geq d_p= 2\mfrak{d}$, then for every $M>0$ such that $B_2(0,M)\subset\paramspace$, there exists $L(M,(t_j)_{j\in[d_o]})>0$ such that
\begin{equation}
\label{eq_locally_Lipschitz_inverse_evaluation_map}
 \forall p,q\in B_2(0,M),\quad \norm{p-q}_2\leq L(M,(t_j)_{j\in[d_o]})\norm{ (s_1(t_j,p))_{j\in[d_o]}-(s_1(t_j,q))_{j\in[d_o]}}_2.
\end{equation}
\end{lemma}
\Cref{lem_locally_Lipschitz_inverse_evaluation_map} states that if the cardinality $d_o$ of the time design $(t_j)_{j\in[d_o]}$ is greater than the parameter dimension $d_p$, then the evaluation map in \eqref{eq_evaluation_map} is invertible and its inverse is locally Lipschitz continuous. 
This claim is reasonable, since if $d_o<d_p$ were true, then we would not expect the evaluation map to be invertible.
To prove the local Lipschitz continuity of the inverse of the evaluation map, we use inverse function theorem, which leads us to the Jacobian of the coefficient map \eqref{eq_Jacobian_coefficient_map}, and to an application of the following result; these steps introduce conditions that imply $d_p=2\mfrak{d}$.
\begin{restatable}{lemma}{NumRootsLinCombExpAffProd}
      \label{lem_number_of_roots_of_linear_combination_of_exponential_affine_products}
     Let $n\in\N$ and $(\beta_k)_{k\in[n]}$, $(\gamma_k)_{k\in[n]}$, $(\xi_k)_{k\in[n]}$ be such that $(\gamma_k,\xi_k)_{k\in[n]}\in\R^{2n}\setminus\{0\}$ and the $(\beta_k)_{k\in[n]}\in\R^n$ are distinct.
     Then the function
     \begin{equation}
     \label{lem_number_of_roots_of_linear_combination_of_exponential_affine_products_function}
        \bb{R} \ni t\mapsto \sum_{k\in[n]}e^{\beta_k t}(\xi_k+\gamma_k t)
     \end{equation}
     has at most $2n-1$ roots.
     \end{restatable}
We prove \Cref{lem_number_of_roots_of_linear_combination_of_exponential_affine_products} in \Cref{appendix_auxiliary_results}.
\begin{proof}[Proof of \Cref{lem_locally_Lipschitz_inverse_evaluation_map}]
Note that if $d_o>d_p$, then 
\begin{equation*}
 \norm{ (s_1(t_j,p))_{j\in[d_p]}-(s_1(t_j,q))_{j\in[d_p]}}_2\leq \norm{ (s_1(t_j,p))_{j\in[d_o]}-(s_1(t_j,q))_{j\in[d_o]}}_2.
\end{equation*}
Thus, it suffices to consider the case where $d_o=d_p$.

Using the product rule and the hypothesis in \Cref{asmp_coefficient_map_Jacobian_full_rank} that $\paramspace\ni p\mapsto (a(p),\lambda (p)$ is $C^1$, it follows from \eqref{eq_ODE_IVP_solution_first_component_exponential_form} that
 \begin{equation}
\forall t\geq 0,\quad  \nabla_p s_1(t,p)
  =\sum_{i\in[\mfrak{d}]} e^{\lambda_i(p)t}\left(\nabla_p a_i(p)+a_i(p) \nabla_p \lambda_i(p) t\right)\in\R^{d_p}.
 \label{eq_Jacobian_ODE_IVP_solution_first_component_exponential_form}
 \end{equation}
Thus, the map $ p\mapsto \evalmap(p)=(s_1(t_j,p))_{j\in[d_o]}$ is $C^1$ on $\paramspace$, and its Jacobian is given by
\begin{equation}
\label{eq_Jacobian_evaluation_map}
 \mscr{J}(q)\coloneqq \begin{bmatrix}
  \nabla_p s_1(t_1,p)\vert \cdots \vert \nabla_p s_1(t_{d_o},p)
 \end{bmatrix}^\top\bigr\vert_{p=q}\in\R^{d_o\times d_p}.
\end{equation}
\noindent
Let $0\neq \alpha\in\R^{d_p}$ and $q\in\paramspace$ be arbitrary.
We shall show that $\mscr{J}(q)\alpha$ is nonzero. 
By the hypothesis in \Cref{asmp_coefficient_map_Jacobian_full_rank} that $\mcal{J}(q)$ defined in \eqref{eq_Jacobian_coefficient_map} has full rank, it follows that if $2\mfrak{d}\geq d_p$, then
\begin{equation*}
 \R^{2\mfrak{d}}\ni \left(\Ang{\nabla_p a_1(p),\alpha},\ldots, \Ang{ \nabla_p a_{\mfrak{d}}(p),\alpha},\Ang{ \nabla_p \lambda_1(p),\alpha},\ldots, \Ang{\nabla_p \lambda_{\mfrak{d}}(p),\alpha}\right)\bigr\vert_{p=q}=\mcal{J}(q)\alpha \neq 0.
\end{equation*}
Combining this with the hypothesis in \Cref{asmp_coefficient_map_Jacobian_full_rank} that $ a(p)\in (\R\setminus\{0\})^{\mfrak{d}}$ for every $p\in\paramspace$, we may apply \Cref{lem_number_of_roots_of_linear_combination_of_exponential_affine_products} with $n\leftarrow \mfrak{d}$, $\beta_k\leftarrow \lambda_k(p)$, $\gamma_k\leftarrow a_k(q)\Ang{\nabla_p \lambda_k(p),\alpha}$, and $\xi_k\leftarrow \Ang{\nabla_p a_k(p),\alpha}$, to conclude that
\begin{equation*}
t\mapsto \Ang{\nabla_p s_1(t,p),\alpha }\vert_{p=q}=\biggr[\sum_{i\in[\mfrak{d}]} e^{\lambda_i(p)t} \bigr(\Ang{\nabla_p a_i(p),\alpha}+a_i(p)\Ang{ \nabla_p \lambda_i(p),\alpha} t\bigr)\biggr]\biggr\vert_{p=q}
\end{equation*}
has at most $2\mfrak{d}-1$ roots in $\R$. 
Thus, if $d_o\geq d_p$ and if $d_p\geq 2\mfrak{d}$, then
\begin{equation}
\label{eq_Jacobian_evaluation_map_has_full_rank}
 \mscr{J}(q) \alpha=\begin{bmatrix}
                     \Ang{\nabla_p s_1(t_1,p),\alpha }\vert \cdots\vert \Ang{\nabla_p s_1(t_{d_o},p),\alpha }
                    \end{bmatrix}^\top \bigr\vert_{p=q}\neq 0.
\end{equation}
Since $\alpha\in\R^{d_p}\setminus\{0\}$ was arbitrary, it follows that $\mscr{J}(q)$ has full rank.
If $d_o=d_p$, then $\mscr{J}(q)$ is invertible, and $\evalmap$ satisfies the hypotheses of the inverse function theorem.

Let $M>0$ be such that $B_2(0,M)\subset \paramspace$.
Recalling from \Cref{ssec_notation} that $\paramspace$ is assumed to be open, it follows that $\closure{B_2(0,M)}\subset \paramspace$.
Since $\evalmap$ is $C^1$ on $\paramspace$ and its Jacobian is invertible on $\closure{B_2(0,M)}$, we may apply the inverse function theorem to conclude that $\evalmap$ is injective in a neighbourhood $U$ of $\closure{B_2(0,M)}$, and its inverse $\evalmap^{-1}$ is a $C^1$ map from $\evalmap(U)$ to $U$.
Since $C^1$ maps are locally Lipschitz, and since $\closure{B_2(0,M)}$ is a compact subset of $U$, it follows that $\evalmap^{-1}$ is Lipschitz on $\closure{\evalmap(B_2(0,M))}$ with Lipschitz constant
\begin{equation*}
\Norm{\evalmap^{-1}\vert_{\closure{\evalmap(B_2(0,M))}}}_{\Lipschitz}\coloneqq\sup\left\{ \frac{\norm{p-q}_2}{\norm{\evalmap(p)-\evalmap(q)}_2}:p,q\in\closure{B_2(0,M)},p\neq q\right\}<\infty.
\end{equation*}
This proves the desired inequality \eqref{eq_locally_Lipschitz_inverse_evaluation_map}, with $L(M,(t_j)_{j\in[d_o]})\coloneqq  \norm{\evalmap^{-1}\vert_{\closure{\evalmap(B_2(0,M))}}}_{\Lipschitz}$.
\end{proof}
In the proof of \Cref{lem_local_boundedness_of_solution_operator}, we obtained an explicit formula for the constant $L(C_1(M),T)$ in \eqref{eq_local_Lipschitz_continuity_of_solution_operator} in terms of $C_1(M)$ and $T$. In contrast, for \Cref{lem_locally_Lipschitz_inverse_evaluation_map}, we defined the  constant $L(M,(t_j)_{j\in[d_o]})$ implicitly, as the Lipschitz constant of the inverse evaluation map. 
Since a closed formula for the inverse of the evaluation map is in general not available, we do not expect that a more explicit definition of $L(M,(t_j)_{j\in[d_o]})$ is available.

\section{Forward operator for fixed time and random covariate design}
\label{sec_forward_operator_fixed_time_random_covariate_design}

In this section, we will use the evaluation map defined in \eqref{eq_evaluation_map} to define a forward operator, and use the results from \Cref{sec_properties_of_solutions_to_linear_ODE_IVPs} to prove properties of the forward operator. This forward operator will determine the Bayesian nonlinear statistical inverse problem that we shall analyse in \Cref{sec_properties_of_Bayesian_inverse_problem}.

We make the following assumption on the coefficient map from  \Cref{def_intrinsic_dimension_of_ODE_IVP_solution_first_component}.
\begin{assumption}
 \label{asmp_applicability_of_observation_function}
 The coefficient map $p\mapsto (a(p),\lambda(p))\in\R^{2\mfrak{d}}$ is such that for every $(t,p)\in[0,T]\times \paramspace$, $s_1(t,p)>0$.
\end{assumption}
By \eqref{eq_ODE_IVP_solution_first_component_exponential_form}, a sufficient condition for \Cref{asmp_applicability_of_observation_function} is that for every $p\in\paramspace$, $a(p)\in\R_{>0}^{\mfrak{d}}$, for example.
In \Cref{sec_application_two_compartment_model} we describe one instance of \eqref{eq_ODE_IVP} for which this sufficient condition holds.

Let $\cpmspace\subset L^2_{\covariatelaw}(\covariatespace,\paramspace)$ be nonempty.
For $(t_j)_{j\in[d_o]}\subset [0,T]$ as in \eqref{eq_evaluation_map},  define the forward operator 
\begin{equation}
\label{eq_fixed_time_random_covariate_design_forward_operator}
 \fwdmodel:\cpmspace\to L^2_{\covariatelaw}(\covariatespace,\R^{d_o}),\quad \cpm(\cdot) \mapsto \fwdmodel(\cpm)(\cdot)\coloneqq(\log s_1(t_j,\cpm(\cdot)))_{j\in[d_o]}.
\end{equation}
The forward operator is obtained by applying the logarithm to every component of the evaluation map $\evalmap$ from \eqref{eq_evaluation_map}.
\Cref{asmp_applicability_of_observation_function} ensures that the forward operator maps CPMs to $\R^{d_o}$-valued functions on $\covariatespace$.
The choice of the logarithm is motivated by \eqref{eq_ODE_IVP_solution_first_component_exponential_form}: for every $p\in\paramspace$, $t\mapsto s_1(t,p)$ is an exponentially growing or decaying function.

Fix a `regularisation space' $(\regspace,\norm{\cdot}_{\regspace})$, i.e. a normed subspace of $\cpmspace$, where
\begin{equation}
 \label{eq_regularisation_space_and_norm}
 \regspace\subseteq L^\infty(\covariatespace,\paramspace)\cap\cpmspace,\quad \norm{f}_{\regspace}\geq \norm{f}_{\infty},\quad B_{\regspace}(M)\coloneqq\{\phi\in\regspace\ :\ \norm{\phi}_{\regspace}< M\}.
 \end{equation}
 We will specify $\regspace$ further in \Cref{sec_properties_of_Bayesian_inverse_problem}. 
 For the results in this section, the properties in \eqref{eq_regularisation_space_and_norm} will suffice.

\begin{proposition}[Local boundedness and Lipschitz continuity of forward operator]
 \label{prop_contraction_in_dG_semimetric}
 Suppose \Cref{asmp_local_boundedness_and_local_lipschitz_continuity_of_matrix_and_IC}, \Cref{asmp_ODE_matrix_diagonalisable} and \Cref{asmp_applicability_of_observation_function} hold. Then for every $M>0$, there exists some $C_2(M,T)$ such that 
  \begin{equation}
  \label{eq_uniform_local_boundedness_of_forward_operator}
  \sup_{\cpm \in \cpmspace\cap B_{\regspace}(M)}\sup_{\covariate\in\covariatespace}\norm{\fwdmodel(\cpm)(\covariate)}_{2}\leq  \sqrt{d_o} C_2(M,T) .
  \end{equation}
  In addition, for the constant $L$ in \eqref{eq_local_Lipschitz_continuity_of_solution_operator}, it holds for every $\covariate\in\covariatespace$ and $\cpm^{(i)}\in B_{\regspace}(M)$, $i=1,2$ that
  \begin{equation}
   \label{eq_pointwise_local_Lipschitz_continuity_of_forward_operator}
    \norm{\fwdmodel(\cpm^{(1)})(\covariate)-\fwdmodel(\cpm^{(2)})(\covariate)}_2\leq \sqrt{d_o}  Le^{C_2(M,T)}\norm{\cpm^{(1)}(\covariate)-\cpm^{(2)}(\covariate)}_2.
  \end{equation}
  Thus, 
  \begin{equation}
   \label{eq_Lq_local_Lipschitz_continuity_of_forward_operator}
    \forall \mu\in\mcal{M}_1(\covariatespace),\ q\in [0,\infty],\quad \norm{\fwdmodel(\cpm^{(1)})-\fwdmodel(\cpm^{(2)})}_{L^q_\mu}\leq \sqrt{d_o} Le^{C_2(M,T)}\norm{\cpm^{(1)}-\cpm^{(2)}}_{L^q_\mu}.
  \end{equation}
\end{proposition}
The conclusion \eqref{eq_Lq_local_Lipschitz_continuity_of_forward_operator} implies that the forward operator $\fwdmodel$ satisfies \cite[Condition 2.1.1]{Nickl2023}, with the choices
\begin{equation}
\label{eq_choices0}
\mcal{Z}\leftarrow\covariatespace,\ W\leftarrow \paramspace,\ V\leftarrow \R^{d_o},\ \zeta\leftarrow \mu,\ \lambda\leftarrow\mu,\  \kappa\leftarrow 0,\ U\leftarrow 1\vee \sqrt{d_o} L(M,T)e^{C_2(M,T)}.
\end{equation}
\begin{proof}[Proof of \Cref{prop_contraction_in_dG_semimetric}]
 Since \Cref{asmp_local_boundedness_and_local_lipschitz_continuity_of_matrix_and_IC} holds, we may apply the conclusions \eqref{eq_local_boundedness_of_solution_operator} and \eqref{eq_local_Lipschitz_continuity_of_solution_operator} of \Cref{lem_local_boundedness_of_solution_operator}.
  Fix an arbitrary $(t,p)\in [0,T]\times B_2(0,M)$. Then
 \begin{equation*}
   (s_0)_1(p)e^{-\norm{A(p)}_2 t}\leq  s_1(t,p)\leq e^{C_1(M)t} C_1(M).
 \end{equation*}
The upper bound follows from \eqref{eq_local_boundedness_of_solution_operator}.
For the lower bound, we use \Cref{asmp_ODE_matrix_diagonalisable} and the definition of the spectral norm $\norm{\cdot}_2$ to bound $\lambda_i(p)\geq -\norm{A(p)}_2$ for every $i\in[\mfrak{d}]$, and then apply \eqref{eq_ODE_IVP_solution_first_component_exponential_form} together with the relation that $\sum_{i\in[\mfrak{d}]}a_i(p)=(s_0)_1(p)$.
By \Cref{asmp_applicability_of_observation_function}, $(s_0)_1(p)=s_1(0,p)>0$.
Using $a\vee b\leq a+b$ for $a,b\geq 0$, we obtain
\begin{equation*}
\sup_{p\in B_2(0,M)}  \abs{\log s_1(t,p)}\leq C_1(M)t+\abs{\log C_1(M)}+\sup_{p\in B_2(0,M)}\biggr(\abs{\log (s_0)_1(p)}+\norm{A(p)}_2 t\biggr).
\end{equation*}
By \eqref{eq_local_bound_on_operator_norm_of_A_and_IC}, $\sup_{p\in B_2(0,M)}\norm{A(p)}_2\leq C_1(M)$.
By \eqref{eq_local_bound_on_operator_norm_of_A_and_IC} and \eqref{eq_local_bound_on_Lipschitz_constant_of_A_and_IC}, $p\mapsto s_0(p)$ is locally bounded, hence so is $p\mapsto (s_0)_1(p)$.
By \Cref{asmp_applicability_of_observation_function} and the definition of the $\norm{\cdot}_2$-norm, $0<(s_0)_1(p)\leq \norm{s_0(p)}_2$ for every $p\in\paramspace$.
Thus, for every $M>0$ there exists some $C'(M)>0$ such that $\sup_{p\in B_2(0,M)}\abs{\log (s_0)_1(p)}\leq C'(M)$, and
\begin{equation}
\label{eq_bound_on_abs_log_first_component_solution}
  \sup_{t\in[0,T]} \sup_{p\in B_2(0,M)}\abs{\log s_1(t,p)}\leq C_1(M)T+\abs{\log C_1(M)}+C'(M) +C_1(M)T\eqqcolon C_2(M,T).
\end{equation}
If $\cpm\in B_\regspace(M)$, then since $\norm{\cdot}_\regspace\geq \norm{\cdot}_\infty$ by \eqref{eq_regularisation_space_and_norm}, it follows that $\cpm(\covariate)\in B_2(0,M)$. 
Thus, by the definition \eqref{eq_fixed_time_random_covariate_design_forward_operator} of $\fwdmodel$, we obtain \eqref{eq_uniform_local_boundedness_of_forward_operator}.

To prove \eqref{eq_pointwise_local_Lipschitz_continuity_of_forward_operator}, we use local Lipschitz continuity of the logarithm \eqref{eq_local_Lipschitz_continuity_logarithm} and the conclusion \eqref{eq_local_Lipschitz_continuity_of_solution_operator} from \Cref{lem_local_boundedness_of_solution_operator}: for arbitrary $p,q\in B_2(0,M)$ and $t\in[0,T]$,
\begin{equation*}
 \abs{\log s_1(t,p)-\log s_1(t,q)}\leq \frac{\abs{s_1(t,p)- s_1(t,q)}}{s_1(t,p)\wedge s_1(t,q)} \leq  \frac{L\norm{p-q}_2}{s_1(t,p)\wedge s_1(t,q)}.
\end{equation*}
By \eqref{eq_bound_on_abs_log_first_component_solution}, $ \inf_{t\in[0,T]} \inf_{p\in B_2(0,M)} s_1(t,p)\geq \exp(-C_2(M,T))$. 
Thus
\begin{equation*}
 \abs{\log s_1(t,p)-\log s_1(t,q)}\leq Le^{C_2(M,T)}\norm{p-q}_2
\end{equation*}
and using the definition \eqref{eq_fixed_time_random_covariate_design_forward_operator} of $\fwdmodel$ then yields 
\begin{equation*}
\forall\covariate\in\covariatespace,\quad \norm{\fwdmodel(\cpm^{(1)})(\covariate)-\fwdmodel(\cpm^{(2)})(\covariate)}_2\leq \sqrt{d_o} L e^{C_2(M,T)}\norm{\cpm^{(1)}(\covariate)-\cpm^{(2)}(\covariate)}_2,
\end{equation*}
which proves \eqref{eq_pointwise_local_Lipschitz_continuity_of_forward_operator}. The inequality \eqref{eq_Lq_local_Lipschitz_continuity_of_forward_operator} follows directly from \eqref{eq_pointwise_local_Lipschitz_continuity_of_forward_operator}.
\end{proof}

Recall the intrinsic dimension $2\mfrak{d}$ from \Cref{def_intrinsic_dimension_of_ODE_IVP_solution_first_component}.
\begin{proposition}[Stability estimate for forward operator]
 \label{prop_stability_estimate}
 Suppose \Cref{asmp_local_boundedness_and_local_lipschitz_continuity_of_matrix_and_IC}, \Cref{asmp_ODE_matrix_diagonalisable}, \Cref{asmp_coefficient_map_Jacobian_full_rank}, and \Cref{asmp_applicability_of_observation_function} hold, and suppose $d_o\geq d_p=2\mfrak{d}$.
Then for every $M>0$, there exists $C_3(M,T,(t_j)_{j\in[d_o]})>0$ such that for any $\cpm^{(i)}\in B_{\regspace}(M)$, $i=1,2$,
\begin{equation*}
 \forall \covariate\in\covariatespace,\quad \Norm{\cpm^{(1)}(\covariate)-\cpm^{(2)}(\covariate)}_2\leq C_3(M,T,(t_j)_{j\in[d_o]})\Norm{\fwdmodel(\cpm^{(1)})(\covariate)-\fwdmodel(\cpm^{(2)})(\covariate)}_2.
\end{equation*}
Thus, 
\begin{equation*}
 \forall \mu\in\mcal{M}_1(\covariatespace),\ q\in [0,\infty],\quad \Norm{\cpm^{(1)}-\cpm^{(2)}}_{L^q_\mu}\leq C_3(M,T,(t_j)_{j\in[d_o]})\Norm{\fwdmodel(\cpm^{(1)})-\fwdmodel(\cpm^{(2)})}_{L^q_\mu}.
\end{equation*}
\end{proposition}
The conclusion of \Cref{prop_stability_estimate} implies that $\fwdmodel$ satisfies \cite[Condition 2.1.4]{Nickl2023} for $L^2_\zeta$ and $L^2_\lambda$ with $\zeta$ and $\lambda$ as in \eqref{eq_choices0}, arbitrary $\cpm_0$ and $\delta$, and 
\begin{equation}
\label{eq_choices1}
 L'=L'_\fwdmodel \leftarrow C_3(M,T,(t_j)_{j\in[d_o]}),\quad \eta\leftarrow 1.
\end{equation}

\begin{proof}[Proof of \Cref{prop_stability_estimate}]
Fix $M>0$. If $\cpm^{(i)}\in B_{\regspace}(M)$, $i=1,2$, then by the properties of the norm $\norm{\cdot}_{\regspace}$ in \eqref{eq_regularisation_space_and_norm}, it follows that for every $\covariate\in\covariatespace$, $\cpm^{(i)}(\covariate)\in B_2(0,M)$.
Since \Cref{asmp_ODE_matrix_diagonalisable} and \Cref{asmp_coefficient_map_Jacobian_full_rank} hold and since $d_o\geq d_p=2\mfrak{d}$, we may apply \Cref{lem_locally_Lipschitz_inverse_evaluation_map}: there exists $L(M,(t_j)_{j\in[d_o]})$ such that for every $\covariate\in\covariatespace$,
 \begin{equation*}
  \Norm{\cpm^{(1)}(\covariate)-\cpm^{(2)}(\covariate)}_2\leq L(M,(t_j)_{j\in[d_o]})\Norm{(s_1(t_j,\cpm^{(1)}(\covariate)))_{j\in[d_o]}-(s_1(t_j,\cpm^{(2)}(\covariate)))_{j\in[d_o]}}_2.
 \end{equation*}
 Since \Cref{asmp_local_boundedness_and_local_lipschitz_continuity_of_matrix_and_IC} holds, we can apply the local Lipschitz continuity of the exponential function \eqref{eq_local_Lipschitz_continuity_exponential} and the bound \eqref{eq_local_boundedness_of_solution_operator} from \Cref{lem_local_boundedness_of_solution_operator}:
 \begin{align*}
  \Abs{s_1(t,p)-s_1(t,q)} \leq&  e^{\log s_1(t,p)\vee \log s_1(t,q)}\abs{\log s_1(t,p)-\log s_1(t,q)}
  \\
  \leq& e^{C_1(M)T}C_1(M)\abs{\log s_1(t,p)-\log s_1(t,q)}.
 \end{align*}
Thus, by definition \eqref{eq_fixed_time_random_covariate_design_forward_operator} of $\fwdmodel$, it holds for every $\covariate\in\covariatespace$ that
\begin{equation*}
\Norm{\cpm^{(1)}(\covariate)-\cpm^{(2)}(\covariate)}_2\leq L(M,(t_j)_{j\in[d_o]}) e^{C_1(M)T}C_1(M)\Norm{\fwdmodel(\cpm^{(1)})(\covariate)-\fwdmodel(\cpm^{(2)})(\covariate)}_2,
\end{equation*}
which yields the first conclusion with $C_3(M,T,(t_j)_{j\in[d_o]})\coloneqq L(M,(t_j)_{j\in[d_o]}) e^{C_1(M)T}C_1(M)$.
By integration, the second conclusion follows.
\end{proof}
Consider the following Jacobian of the componentwise logarithm of the evaluation map:
\begin{equation}
 \label{eq_Jacobian_forward_model}
 \paramspace\ni q\mapsto \mfrak{J}(q)\coloneqq 
 \begin{bmatrix}
   \nabla_p \log s_1(t_1,p)~\vert \cdots\vert~ \nabla_p \log s_1(t_{d_o},p)
  \end{bmatrix}^\top\bigr\vert_{p=q}\in\R^{d_o\times d_p}.
\end{equation}
We will use $\mfrak{J}(\cdot)$ to further analyse the forward operator $\fwdmodel$.
\begin{lemma}
 \label{lem_invertible_transpose_Jacobian}
 Suppose \Cref{asmp_ODE_matrix_diagonalisable}, \Cref{asmp_coefficient_map_Jacobian_full_rank}, and \Cref{asmp_applicability_of_observation_function} hold.
 If $d_o\geq d_p= 2\mfrak{d}$, then for every $q\in\paramspace$, $\mfrak{J}(q)$ has full rank.
\end{lemma}
\begin{proof}[Proof of \Cref{lem_invertible_transpose_Jacobian}]
     Given the hypotheses, we may apply statements in the proof of  \Cref{lem_locally_Lipschitz_inverse_evaluation_map}.
     
     By the chain rule, $\nabla_p \log s_1(t,p)=(s_1(t,p))^{-1}\nabla_p s_1(t,p)$, for every $(t,p)\in [0,T]\times\paramspace$.
     Thus, the Jacobian $\mscr{J}(q)$ defined in \eqref{eq_Jacobian_evaluation_map} and $\mfrak{J}(q)$ in \eqref{eq_Jacobian_forward_model} satisfy
     \begin{equation*}
      \forall q\in\paramspace,\quad \mfrak{J}(q)=\begin{bmatrix}
       s_1(t_1,q) & 0 & \cdots & 0
       \\
       0 & s_1(t_2,q) & \ddots & \vdots
       \\
       \vdots & \ddots & \ddots & 0
       \\
       0 & \cdots & 0 & s_1(t_{d_o},q)        
      \end{bmatrix}^{-1} \mscr{J}(q),
     \end{equation*}
     where the existence of the inverse matrix on the right-hand side is guaranteed by \Cref{asmp_applicability_of_observation_function}.
     By \eqref{eq_Jacobian_evaluation_map_has_full_rank}, $\mscr{J}(q)\in\R^{d_o\times d_p}$ has full rank. Thus, $\mfrak{J}(q)$ also has full rank.
\end{proof}
The following assumption strengthens the $C^1$ condition of the coefficient map in \Cref{asmp_coefficient_map_Jacobian_full_rank}.
\begin{assumption}
 \label{asmp_coefficient_map_C2}
The coefficient map $\paramspace\ni p\mapsto (a(p),\lambda(p))$ is $C^2$. 
\end{assumption}
The following assumption constrains a reference CPM $\cpm_0$. The assumption holds if $\norm{\cpm_0}_\infty$ is finite, for example.
\begin{assumption}
\label{asmp_image_of_cpm0_contained_in_compact_subset}
There exists a compact subset $K$ of $\paramspace$ such that $\{\cpm_0(\covariate):\covariate\in\covariatespace\}\subset K$.
\end{assumption}
The next result describes properties of the derivative of the forward operator.
\begin{proposition}
 \label{prop_derivative_forward_operator}
 Suppose \Cref{asmp_ODE_matrix_diagonalisable}, \Cref{asmp_coefficient_map_Jacobian_full_rank}, and \Cref{asmp_applicability_of_observation_function} hold.
 Let $\cpm_0\in\cpmspace$ and let $H\subset L^\infty(\covariatespace,\paramspace)$ be a linear space such that for some $\epsilon>0$, $\{\cpm_0+h:h\in H,\ \norm{h}_\infty<\epsilon\}\subset \cpmspace$.
 Define a linear operator $\bb{I}_{\cpm_0}$ acting on $H$ according to
 \begin{equation}
  \label{eq_derivative_forward_operator}
  H\ni h(\cdot)\mapsto \bb{I}_{\cpm_0}[h](\cdot)\coloneqq 
\mfrak{J}(\cpm_0(\cdot))h(\cdot).
 \end{equation}
 \begin{enumerate}
 \item  \label{item_derivative_forward_operator_asymptotics}
 For any $\mu\in\mcal{M}_1(\covariatespace)$,
 \begin{equation}
 \label{eq_derivative_asymptotic_behaviour_of_error_term}
  \rho_{\cpm_0}[h]\coloneqq \Norm{\fwdmodel(\cpm_0+h)-\fwdmodel(\cpm_0)-\bb{I}_{\cpm_0}[h]}_{L^2_\mu}=o(\norm{h}_\infty),\quad \norm{h}_\infty\to 0.
 \end{equation}
 \item If $d_o\geq d_p= 2\mfrak{d}$, then $\bb{I}_{\cpm_0}$ is injective.
\label{item_derivative_forward_operator_injectivity}
 
 \item If $\sum_{j\in[d_o]}\norm{\nabla_p\log s_1(t_j,p)\vert_{p=\cpm_0}}^2_{L^2_\mu}$ is finite, then $\bb{I}_{\cpm_0}:(H,\Ang{\cdot,\cdot}_{L^2_\mu(\covariatespace,\paramspace)})\to L^2_\mu(\covariatespace,\R^{d_o})$ is continuous. 
 \label{item_derivative_forward_operator_continuity}
 
\item If in addition \Cref{asmp_coefficient_map_C2} and \Cref{asmp_image_of_cpm0_contained_in_compact_subset} hold, then $\rho_{\cpm_0}[h]=O(\norm{h}_\infty^2)$.
\label{item_derivative_forward_operator_asymptotics_quadratic}
\end{enumerate}
\end{proposition}
Statements \ref{item_derivative_forward_operator_asymptotics}--\ref{item_derivative_forward_operator_continuity} of \Cref{prop_derivative_forward_operator} imply that $\fwdmodel$ satisfies \cite[Condition 3.1.1]{Nickl2023}, for the spaces $L^2_\zeta(\mcal{Z},W)$ and $L^2_\lambda(\mcal{X},V)$ defined by the choices in \eqref{eq_choices0}.
Statement \ref{item_derivative_forward_operator_asymptotics_quadratic} of \Cref{prop_derivative_forward_operator} implies that \cite[Condition 4.1.4]{Nickl2023} holds under simpler conditions; see the discussion after equation (4.6) on \cite[p.70]{Nickl2023}.
We will use this fact in the proof of \Cref{thm_Bernstein_von_Mises_for_linear_functionals}.
\begin{proof}[Proof of \Cref{prop_derivative_forward_operator}]
Let $\cpm_0$ and $H$ be as given, and fix an arbitrary $h\in H$.

By \Cref{asmp_coefficient_map_Jacobian_full_rank}, $\paramspace\ni p\mapsto (a(p),\lambda(p))$ is $C^1$ smooth.
By applying Taylor's theorem to the map $p\mapsto \log s_1(t,p)$, we have for every $(t,\covariate)\in[0,T]\times\covariatespace$ that
\begin{equation*}
\log s_1(t,\cpm_0(\covariate)+h(\covariate))=\log s_1(t,\cpm_0(\covariate))+\Ang{\nabla_p\log s_1(t,p)\vert_{p=\cpm_0(\covariate)},h(\covariate)}+o(\norm{h(\covariate)}_2).
\end{equation*}
By the definitions \eqref{eq_fixed_time_random_covariate_design_forward_operator} and \eqref{eq_derivative_forward_operator} of $\fwdmodel$ and $\bb{I}_{\cpm_0}$ respectively, this implies that 
\begin{equation*}
\forall \covariate\in\covariatespace,\quad \norm{\fwdmodel(\cpm_0+h)(\covariate)-\fwdmodel(\cpm_0)(\covariate)-\bb{I}_{\cpm_0}[h](\covariate)}_2=o(\norm{h(\covariate)}_2).
\end{equation*}
 Fix an arbitrary $\mu\in\mcal{M}_1(\covariatespace)$.
It follows from the definition \eqref{eq_derivative_asymptotic_behaviour_of_error_term} of $\rho_{\cpm_0}[h]$, the preceding fact, the dominated convergence theorem, and the fact that $\norm{h(\covariate)}_2\leq \norm{h}_\infty$ for every $\covariate$, that
\begin{equation*}
\lim_{\norm{h}_\infty\to 0}\frac{\rho_{\cpm_0}[h]^2}{\norm{h}_\infty^2}= \lim_{\norm{h}_\infty\to 0}\frac{ \Norm{\fwdmodel(\cpm_0+h)-\fwdmodel(\cpm_0)-\bb{I}_{\cpm_0}[h]}^2_{L^2_\mu} }{\norm{h}^2_\infty}
 = \int_\covariatespace \lim_{\norm{h}^2_\infty\to 0}\frac{o(\norm{h(\covariate)}^2_2)}{\norm{h}^2_\infty}\mu(\rd \covariate) =0.
\end{equation*}
This proves statement \ref{item_derivative_forward_operator_asymptotics}.

If $d_o\geq d_p=2\mfrak{d}$, then we may apply  \Cref{lem_invertible_transpose_Jacobian} to conclude that $\mfrak{J}(\cpm_0(\covariate))$ has full rank for every $\covariate\in\covariatespace$. 
Thus, by \eqref{eq_derivative_forward_operator}, $\bb{I}_{\cpm_0}[h_1-h_2]=0$ if and only if $h_1-h_2=0$. This proves statement \ref{item_derivative_forward_operator_injectivity}.

By \eqref{eq_derivative_forward_operator}, \eqref{eq_Jacobian_forward_model}, and the Cauchy--Schwarz inequality,
\begin{equation}
\label{eq_item_derivative_forward_operator_continuity}
\Norm{\bb{I}_{\cpm_0}[h]}^2_{L^2_\mu}=\sum_{j\in[d_o]} \Norm{\Ang{\nabla_p\log s_1(t_j,p)\vert_{p=\cpm_0}, h}}^2_{L^2_\mu}
 \leq \sum_{j\in[d_o]}\Norm{\nabla_p\log s_1(t_j,p)\vert_{p=\cpm_0}}^2_{L^2_\mu}\Norm{h}^2_{L^2_\mu},
\end{equation}
which proves statement \ref{item_derivative_forward_operator_continuity}.

If \Cref{asmp_coefficient_map_C2} holds, then by Taylor's theorem, it holds for any $(t,\covariate)\in[0,T]\times\covariatespace$ that
\begin{equation*}
\abs{\log s_1(t,(\cpm_0+h)(\covariate))-\log s_1(t,\cpm_0(\covariate))-\Ang{\nabla_p\log s_1(t,p)\vert_{p=\cpm_0(\covariate)},h(\covariate)}}\leq C(\cpm_0,t,x)\norm{h(\covariate)}_2^2,
\end{equation*}
where $C(\cpm_0,t,x)\lesssim \norm{\nabla_p^2 \log s_1(t,p)\vert_{p=\cpm_0(\covariate)}}_2$.
Let $K$ be the compact set described in \Cref{asmp_image_of_cpm0_contained_in_compact_subset}. 
Define $C(K)>0$ by 
\begin{equation*}
\max_{j\in[d_o]}\sup_{\covariate\in\covariatespace}\Norm{\nabla_p^2 \log s_1(t_j,p)\vert_{p=\cpm_0(\covariate)}}_2\leq \max_{j\in[d_o]}\sup_{q\in K}\Norm{\nabla_p^2 \log s_1(t_j,p)\vert_{p=q}}_2\eqqcolon C(K).
\end{equation*}
By applying the extreme value theorem to the continuous maps $p\mapsto \norm{\nabla^2_p\log s_1(t_j,p)}_2$, $j\in[d_o]$, it follows that $C(K)$ is finite.
Thus,
\begin{equation*}
\forall \covariate\in\covariatespace,\quad \norm{\fwdmodel(\cpm_0+h)(\covariate)-\fwdmodel(\cpm_0)(\covariate)-\bb{I}_{\cpm_0}[h](\covariate)}_2\lesssim C(K)\norm{h(\covariate)}_2\leq C(K)\norm{h}_\infty^2,
\end{equation*}
and by the definition \eqref{eq_derivative_asymptotic_behaviour_of_error_term} of $\rho_{\cpm_0}[h]$, statement \ref{item_derivative_forward_operator_asymptotics_quadratic} follows.
\end{proof}
The next result states the local Lipschitz continuity in supremum norm of the forward operator.
\begin{proposition}
 \label{prop_derivative_forward_operator_Linfty}
 Suppose that \Cref{asmp_local_boundedness_and_local_lipschitz_continuity_of_matrix_and_IC}, \Cref{asmp_ODE_matrix_diagonalisable}, \Cref{asmp_coefficient_map_Jacobian_full_rank}, and \Cref{asmp_applicability_of_observation_function} hold.
 Let $\cpmspace=H=L^\infty(\covariatespace,\paramspace)$ and $\cpm_0\in\cpmspace$. Then for every $M>0$ there exists $L'>0$ such that for every $\cpm^{(i)}$ with $\norm{\cpm^{(i)}}_{\regspace}\leq M$ for $i=1,2$, 
 \begin{equation*}
  \Norm{\fwdmodel(\cpm^{(1)})-\fwdmodel(\cpm^{(2)})}_\infty\leq L'\Norm{\cpm^{(1)}-\cpm^{(2)}}_\infty.
 \end{equation*}
In addition, $\bb{I}_{\cpm_0}:(H,\norm{\cdot}_\infty)\to L^\infty(\covariatespace,\R^{d_o})$ is continuous.
\end{proposition}
The conclusions of \Cref{prop_derivative_forward_operator_Linfty} imply that \cite[Condition 4.1.1]{Nickl2023} holds for $H\leftarrow L^\infty(\covariatespace,\paramspace)$ and $V\leftarrow \R^{d_o}$.
\begin{proof}[Proof of \Cref{prop_derivative_forward_operator_Linfty}]
 Since \Cref{asmp_local_boundedness_and_local_lipschitz_continuity_of_matrix_and_IC}, \Cref{asmp_ODE_matrix_diagonalisable}, and \Cref{asmp_applicability_of_observation_function} hold, we may apply \Cref{prop_contraction_in_dG_semimetric}; \eqref{eq_pointwise_local_Lipschitz_continuity_of_forward_operator} implies the desired bound, with $L'\leftarrow \sqrt{d_o}L(M,T)e^{C_2(M,T)}$.
 
To prove the continuity statement, it suffices to show that the hypothesis of statement \ref{item_derivative_forward_operator_continuity} of \Cref{prop_derivative_forward_operator} holds.
 By the definition \eqref{eq_derivative_forward_operator} of $\bb{I}_{\cpm_0}$, we have for every $\covariate\in\covariatespace$ that
 \begin{align*}
  \norm{\bb{I}_{\cpm_0}[h](\covariate)}_2^2=&\sum_{j\in[d_o]} \Abs{\Ang{\nabla_p\log s_1(t_j,p)\vert_{p=\cpm_0(\covariate)}, h(\covariate)}}^2
  \\
  \leq & \sum_{j\in[d_o]} \Norm{\nabla_p\log s_1(t_j,p)\vert_{p=\cpm_0(\covariate)}}^2_2\Norm{h(\covariate)}^2_2
  \\
  \leq & \sup_{y\in\covariatespace}\sum_{j\in[d_o]} \Norm{\nabla_p\log s_1(t_j,p)\vert_{p=\cpm_0(y)}}^2_2 \Norm{h(\covariate)}^2_2.
 \end{align*}
By \Cref{asmp_coefficient_map_Jacobian_full_rank}, we may apply \eqref{eq_Jacobian_ODE_IVP_solution_first_component_exponential_form} and the chain rule to conclude that the map $q\mapsto \nabla_p\log s_1(t_j,p)\vert_{p=q}$ is continuous. 
Since $\cpm_0\in\cpmspace=H= L^\infty(\covariatespace,\paramspace)$, the Heine--Borel theorem implies that $K\coloneqq \closure{\{\cpm_0(\covariate):\covariate\in\covariatespace\}}$ is a compact subset of $\paramspace$. Thus, by the extreme value theorem, 
\begin{equation}
\label{eq_item_derivative_forward_operator_continuity_Linfty}
 \sup_{y\in\covariatespace}\sum_{j\in[d_o]} \Norm{\nabla_p\log s_1(t_j,p)\vert_{p=\cpm_0(y)}}^2_2\leq \sup_{q\in K}\sum_{j\in[d_o]} \norm{\nabla_p\log s_1(t_j,p)\vert_{p=q}}^2_2<\infty,
\end{equation}
and combining the previous steps yields the continuity of $\bb{I}_{\cpm_0}:(H,\norm{\cdot}_\infty)\to L^\infty(\covariatespace,\R^{d_o})$.
\end{proof}

 Let $\mu\in\mcal{M}_1(\covariatespace)$. 
 Given $\mfrak{J}:\covariatespace\to \R^{d_p\times d_o}$ and $\bb{I}_{\cpm_0}[h](\cdot)=\mfrak{J}(\cpm_0(\cdot))h(\cdot)$ from \eqref{eq_Jacobian_forward_model} and \eqref{eq_derivative_forward_operator} respectively, the operator
    \begin{equation}
     \bb{I}_{\cpm_0}^{\ast}:L^2_{\mu}(\covariatespace,\R^{d_o})\to  \overline{(H,\ang{\cdot,\cdot}_{L^{2}_{\mu}(\covariatespace,\paramspace)})}\vert_{L^{2}_{\mu}},
     \qquad 
     g(\cdot)\mapsto \bb{I}_{\cpm_0}^{\ast}[g](\cdot)\coloneqq (\mfrak{J}(\cpm_0(\cdot))^\top g(\cdot)
     \label{eq_derivative_forward_operator_adjoint}
    \end{equation}
    is the Hilbert space adjoint of $\bb{I}_{\cpm_0}$, because
    \begin{align*}
     \Ang{\bb{I}_{\cpm_0}[h],g}_{L^2_{\mu}(\covariatespace,\R^{d_o})}=&\int_{\covariatespace} \Ang{\mfrak{J}(\cpm_0(\covariate))h(\covariate),g(\covariate)}_{\R^{d_o}}\mu(\rd x) & 
     \\
     =& \int_{\covariatespace} \Ang{ h(\covariate),\mfrak{J}(\cpm_0(\covariate))^\top}_{\R^{d_p}}\mu(\rd x) =\Ang{h,\bb{I}_{\cpm_0}^{\ast}[g]}_{L^2_{\mu}(\covariatespace,\paramspace)}.
    \end{align*}
    Following \cite[Definition 3.1.2]{Nickl2023}, we define the information operator corresponding to $\fwdmodel$ at $\cpm_0$ and a measure $\mu\in\mcal{M}_1(\covariatespace)$ by
    \begin{equation}
     \informoperator: (H,\ang{\cdot,\cdot}_{L^2_{\mu}(\covariatespace,\paramspace)})\to \overline{(H,\ang{\cdot,\cdot}_{L^2_{\mu}(\covariatespace,\paramspace)})}\vert_{L^2_{\mu}(\covariatespace,\paramspace)},\quad
    \informoperator[h](\covariate)\coloneqq \mfrak{J}^\top\mfrak{J}(\cpm_0(\covariate))h(\covariate).
     \label{eq_information_operator}
    \end{equation}

    \begin{proposition}
     \label{prop_solvability_information_equation}
     Suppose \Cref{asmp_ODE_matrix_diagonalisable}, \Cref{asmp_coefficient_map_Jacobian_full_rank}, and \Cref{asmp_applicability_of_observation_function} hold, and $d_o\geq d_p= 2\mfrak{d}$.
     Let $\cpm_0\in L^\infty(\covariatespace,\paramspace)$.
     Then for every $\psi\in L^\infty(\covariatespace,\paramspace)$, there exists a unique $\overline{\psi}_{\cpm_0}\in L^\infty(\covariatespace,\paramspace)$ such that $\informoperator[\overline{\psi}_{\cpm_0}]=\psi$ on $\covariatespace$.
     In particular, for all $h\in L^\infty(\covariatespace,\paramspace)$, $\Ang{\informoperator[\overline{\psi}]-\psi,h}_{L^2_\mu(\covariatespace,\paramspace)}=0$.
    \end{proposition}
    The conclusion of \Cref{prop_solvability_information_equation} implies that \cite[Condition 4.1.2]{Nickl2023} holds.
    \begin{proof}[Proof of \Cref{prop_solvability_information_equation}]
     By the hypotheses, we may apply \Cref{lem_invertible_transpose_Jacobian}: for every $\covariate\in\covariatespace$, $\mfrak{J}(\cpm_0(\covariate))\in \R^{d_o\times d_p}$ has full rank, and thus $\mfrak{J}^\top\mfrak{J}(\cpm_0(\covariate))\in GL(d_p,\R)$.
     Given $\cpm_0,\psi\in L^\infty(\covariatespace,\paramspace)$, define $\overline{\psi}_{\cpm_0}$ by
     \begin{equation*}
      \forall\covariate\in\covariatespace,\quad \overline{\psi}_{\cpm_0}(\covariate)\coloneqq \left(\mfrak{J}^\top\mfrak{J}(\cpm_0(\covariate))\right)^{-1}\psi(\covariate).
     \end{equation*}
     By \eqref{eq_information_operator}, $\informoperator[\overline{\psi}_{\cpm_0}]=\psi$ on $\covariatespace$.
     We prove $\overline{\psi}_{\cpm_0}\in L^\infty(\covariatespace,\paramspace)$, following the same strategy as the proof of \Cref{prop_derivative_forward_operator_Linfty}. Since $\cpm_0\in L^\infty(\covariatespace,\paramspace)$ by hypothesis, it follows by the Heine--Borel theorem that $K\coloneqq\closure{\{\cpm_0(\covariate):\covariate\in\covariatespace\}}$ is compact, and     
     \begin{equation*}
      \sup_{\covariate\in\covariatespace}\Norm{\bigr(\mfrak{J}^\top\mfrak{J}(\cpm_0(\covariate))\bigr)^{-1}}_{2}\leq \sup_{q\in K}\Norm{ \bigr(\mfrak{J}^\top\mfrak{J}(q)\bigr)^{-1}}_{2}.
     \end{equation*}
     By \Cref{asmp_coefficient_map_Jacobian_full_rank} and \eqref{eq_Jacobian_forward_model}, $q\mapsto \mfrak{J}^\top\mfrak{J}(q)$ is continuous.
     Thus, by the extreme value theorem, the right-hand side of the inequality above is finite. 
     By the definition of $\overline{\psi}_{\cpm_0}$, we have $\norm{\overline{\psi}_{\cpm_0}}_\infty\leq \norm{ \mfrak{J}^\top_{\cpm_0}\mfrak{J}(\cpm_0)^{-1}}_\infty\norm{\psi}_\infty$, and $\psi\in L^\infty(\covariatespace,\paramspace)$ implies $\overline{\psi}_{\cpm_0}\in L^\infty(\covariatespace,\paramspace)$.
     \end{proof}

    \section{The Bayesian inverse problem and properties of its solution}
    \label{sec_properties_of_Bayesian_inverse_problem}

    Recall the definition \eqref{eq_fixed_time_random_covariate_design_forward_operator} of $\fwdmodel$, and consider the following observation model: for $k\in\N$, the $k$-th observation is a random variable $Y^{(k)}$ given by
\begin{equation}
 \label{eq_fixed_time_random_covariate_design_observation_model}
\R^{d_o}\ni Y^{(k)}=\fwdmodel(\cpm_0)(X^{(k)})+\varepsilon^{(k)},\quad \varepsilon^{(k)}\distiid \normaldist{0}{\varsigma^2 I},\quad X^{(k)}\distiid \covariatelaw,
\end{equation}
where $\varsigma>0$ is presumed to be known, and $\covariatelaw\in\mcal{M}_1(\covariatespace)$ is not known.
In this section, we consider the following statistical inverse problem: given pairs $(Y^{(k)},X^{(k)})$, $k\in\N$, infer the true data-generating CPM $\cpm_0$.
We apply the Bayesian approach, by specifying a family of Gaussian priors and analysing the properties of the associated family of posterior laws.
We emphasise that while the ODE-IVP \eqref{eq_ODE_IVP} that defines the forward operator $\fwdmodel$ is linear, the dependence of the solution of \eqref{eq_ODE_IVP} --- in particular, of the observed component \eqref{eq_ODE_IVP_solution_first_component_exponential_form} of the solution --- on the parameter $p$ is nonlinear.
    Thus, the forward operator $\fwdmodel$ and the resulting inverse problem are nonlinear.
        
    Recall from \Cref{ssec_notation} that $d_\covariate=\dimension(\covariatespace)$.
    For $r\geq 0$, let $H^r(\covariatespace,\paramspace)$ denote the Sobolev space of regularity $r$.
    \begin{assumption}[$\alpha$-smooth Gaussian base prior]
     \label{asmp_Gaussian_base_prior}
     Let $\Pi'$ be a centred Gaussian Borel probability measure on $\cpmspace\subseteq L^2_{\covariatelaw}(\covariatespace,\paramspace)$ with reproducing kernel Hilbert space (RKHS) $\mcal{H}$ that is continuously embedded in $H^\alpha(\covariatespace,\paramspace)$ for some $\alpha>0$, and $\Pi'(\regspace)=1$ for a separable normed linear subspace $(\regspace,\norm{\cdot}_{\regspace})$ of $\cpmspace$.
    \end{assumption}
    The scalar $\alpha>0$ describes the `smoothness' of the base prior $\Pi'$.
    \Cref{asmp_Gaussian_base_prior} implies that \cite[Condition 2.1.1]{Nickl2023} holds.
    The condition that $\mcal{H}$ is continuously embedded in $H^\alpha$ is from \cite[Theorem 2.2.2]{Nickl2023}, and is due to the choice $\kappa\leftarrow 0$ in \eqref{eq_choices0}.
    Define a sequence $(\Pi_N)_{N\in\N}$ of Gaussian priors by rescaling the Gaussian base prior $\Pi'$ according to the sample size $N$:
    \begin{equation}
    \label{eq_rescaled_priors_Pi_N}
      \cpm^{(N)}\coloneqq N^{-d_\covariate/(4\alpha+2 d_\covariate )}\theta',\quad \theta'\sim\Pi',\quad \Pi_N\coloneqq \law{\cpm^{(N)}}.
    \end{equation}
    The definitions \eqref{eq_rescaled_priors_Pi_N} follow from \cite[Equation (2.18)]{Nickl2023} with the replacement $\kappa\leftarrow 0$ in \eqref{eq_choices0}.
    For $N\in\N$, $\Pi_N(\cdot~\vert (Y^{(k)},X^{(k)})_{k\in[N]})$ denotes the posterior corresponding to $(Y^{(k)},X^{(k)})_{k\in[N]}$ and the $N$-th prior $\Pi_N$.
    
    Next, define a sequence $(\delta_N)_{N\in\N}$ of positive numbers and a sequence $(\cpmspace_{N})_{N\in\N}$ of regularisation sets in $\regspace$, following \cite[Equations (2.19) and (2.20)]{Nickl2023}:
    \begin{subequations}
     \begin{align}
      \delta_{N}\coloneqq & N^{-\alpha/(2\alpha+ d_\covariate )},
      \label{eq_delta_N_rate}
      \\
      \cpmspace_N\coloneqq & \left\{\cpm_1+\cpm_2 \in\regspace\ :\ \norm{\cpm_1}_{L^2_{\covariatelaw}}\leq M\delta_N,\ \norm{\cpm_2}_{\mcal{H}}\leq M,\ \norm{\cpm_1+\cpm_2}_{\regspace}\leq M\right\}.
      \label{eq_Theta_N_regularisation_sets}
     \end{align}
    \end{subequations}
    We now state sufficient conditions for the posteriors $(\Pi_N(\cdot~\vert (Y^{(k)},X^{(k)})_{k\in[N]}))_{N\in\N}$ to contract to the truth $\cpm_0$ with rate $(\delta_N)_{n\in\N}$.
    \begin{theorem}[Posterior contraction in $L^2_{\covariatelaw}$ norm]
    \label{thm_posterior_contraction_L2_norm}
    Suppose that $\covariatespace$ is a bounded domain of $\R^{d_\covariate}$ with smooth boundary, and that the hypotheses of \Cref{prop_stability_estimate} hold. Suppose \Cref{asmp_Gaussian_base_prior} holds for $\regspace$ given in \eqref{eq_regularisation_space_and_norm} and some $\alpha>0$.
    If $\cpm_0\in\mcal{H}\cap\regspace$, then for all $b>0$, we can choose $m>0$ large enough such that 
     \begin{equation*}
       P^{N}_{\cpm_0}\left(\Pi_N(\theta\in\Theta_N: \norm{\fwdmodel(\theta)-\fwdmodel(\theta_0)}_{L^2_{\covariatelaw}}\leq m\delta_N\vert (Y^{(k)},X^{(k)})_{k\in[N]})\leq 1-e^{bN\delta_N^2}\right)\xrightarrow[N\to\infty]{} 0.
     \end{equation*}
    In addition, for $L'$ in \eqref{eq_choices1},
    \begin{equation*}
     P^{N}_{\cpm_0}\left(\Pi_N(\theta\in\Theta_N : \norm{\theta-\theta_0}_{L^2_{\covariatelaw}}\leq L' m\delta_N\vert (Y^{(k)},X^{(k)})_{k\in[N]})\leq 1-e^{bN\delta_N^2}\right)\xrightarrow[N\to\infty]{} 0,
    \end{equation*}
    and the sequence of posterior means $(E_{\Pi_N}[\cpm\vert(Y^{(k)},X^{(k)})_{k\in[N]}])_{N\in\N}$ satisfies
    \begin{equation*}
     \Norm{E_{\Pi_N}[\cpm\vert (Y^{(k)},X^{(k)})_{k\in[N]}]-\cpm_0}_{L^2_{\covariatelaw}}=O_{P^N_{\cpm_0}}(\delta_N).
    \end{equation*}
    \end{theorem}
    \begin{proof}[Proof of \Cref{thm_posterior_contraction_L2_norm}]
     The hypotheses of \Cref{prop_stability_estimate} include those of \Cref{prop_contraction_in_dG_semimetric}. Thus, $\fwdmodel$ satisfies \cite[Condition 2.1.1]{Nickl2023} for the choices in \eqref{eq_choices0}.
     Since $\covariatespace$ is a bounded domain with smooth boundary, $\kappa\leftarrow 0$ in \eqref{eq_choices0}, and \Cref{asmp_Gaussian_base_prior} holds, the hypotheses of \cite[Theorem 2.2.2]{Nickl2023} are satisfied.
     Since the hypotheses of \Cref{prop_stability_estimate} hold, we conclude that $\fwdmodel$ satisfies \cite[Condition 2.1.4]{Nickl2023} for the choices in \eqref{eq_choices0} and \eqref{eq_choices1}.
     The posterior contraction result then follows from \cite[Theorem 2.3.1]{Nickl2023}.
     The stochastic bound on the error of the posterior mean follows from \cite[Theorem 2.3.2]{Nickl2023}.
    \end{proof}

     We write `$\xrightarrow[]{d}$' to denote convergence in distribution under $P^{\N}_{\cpm_0}$ of $\R$-valued random variables. 
     We now state sufficient conditions for the observation model \eqref{eq_fixed_time_random_covariate_design_observation_model} to satisfy a locally asymptotically normal (LAN) approximation.
    \begin{theorem}
     \label{thm_local_asymptotic_normality}
     Suppose the hypotheses of \Cref{prop_derivative_forward_operator_Linfty} hold, so that $\cpmspace=H=L^\infty(\covariatespace,\paramspace)$, and suppose $d_o\geq d_p=2\mfrak{d}$. Fix an arbitrary $h\in H$.
     If $\cpm_0\in \cpmspace$, then the log-likelihood ratio process in the model \eqref{eq_fixed_time_random_covariate_design_observation_model} 
     satisfies the LAN-approximation 
     \begin{equation*}
      \log \frac{\rd P^N_{\cpm_0+h/\sqrt{N}}}{\rd P^N_{\cpm_0}}((Y^{(k)},X^{(k)})_{k\in[N]})\xrightarrow[N\to\infty]{d} \normaldist{-\tfrac{1}{2}\norm{\bb{I}_{\cpm_0}[h]}^{2}_{L^2_{\covariatelaw}(\covariatespace,\R^{d_o})}}{\norm{\bb{I}_{\cpm_0}[h]}^{2}_{L^2_{\covariatelaw}(\covariatespace,\R^{d_o})}}.
     \end{equation*}
    \end{theorem}
    \begin{proof}[Proof of \Cref{thm_local_asymptotic_normality}]
     The assumptions of \Cref{prop_derivative_forward_operator_Linfty} include the assumptions of statements \ref{item_derivative_forward_operator_asymptotics}--\ref{item_derivative_forward_operator_injectivity} of \Cref{prop_derivative_forward_operator}. 
     Given the hypotheses, we may also apply \eqref{eq_item_derivative_forward_operator_continuity_Linfty} to \eqref{eq_item_derivative_forward_operator_continuity}, which implies that statement \ref{item_derivative_forward_operator_continuity} of \Cref{prop_derivative_forward_operator} holds as well.
     Statements \ref{item_derivative_forward_operator_asymptotics}--\ref{item_derivative_forward_operator_continuity} of \Cref{prop_derivative_forward_operator} imply that \cite[Condition 3.1.1]{Nickl2023} is satisfied.
     
     Next, fix an arbitrary $\hat{s}>0$.
     Given the stated hypotheses, we can apply \Cref{prop_derivative_forward_operator_Linfty} with $\cpm^{(1)}\leftarrow \cpm_0$, $\cpm^{(2)}\leftarrow \cpm_0+\hat{s}h$, and $M\leftarrow \max\{\norm{\cpm_0+sh}_\infty: -\hat{s}\leq s\leq \hat{s}\}$ to conclude that there exists $L=L(M)>0$ such that
     \begin{equation*}
      \forall s\in [-\hat{s},\hat{s}],\quad \norm{\fwdmodel(\cpm_0+sh)-\fwdmodel(\cpm_0)}_\infty\leq L \norm{\cpm_0+sh-\cpm_0}_\infty=L\abs{s}\norm{h}_\infty\leq L \hat{s}\norm{h}_\infty.
     \end{equation*}
     Thus, we may apply \cite[Theorem 3.1.3]{Nickl2023} to obtain the desired conclusion.
    \end{proof}
    For the following theorem, recall the definition \eqref{eq_derivative_forward_operator_adjoint} of $\bb{I}_{\cpm_0}^\ast$, and fix an arbitrary $\widehat{\psi}_{\cpm_0}\in L^2_{\covariatelaw}(\covariatespace,\R^{d_o})$. 
    If $\cpmspace=H=L^\infty(\covariatespace,\paramspace)$, define 
     \begin{equation*}
      \Psi:H\to\R,\quad h\mapsto \Psi(h)\coloneqq \Ang{\bb{I}_{\cpm}[h],\widehat{\psi}_{\cpm_0}}_{L^2_{\covariatelaw}(\covariatespace,\paramspace)}.
     \end{equation*}
    The next result presents a local asymptotic minimax bound for linear functionals.
     \begin{theorem}
    \label{thm_local_asymptotic_minimax_bound_linear_functionals}
       Suppose that the hypotheses of \Cref{thm_local_asymptotic_normality} hold.
     Then the local asymptotic minimax risk for estimating $\Psi$ at $\cpm_0$ satisfies
     \begin{equation*}
      \liminf_{N\to\infty}~\inf_{\overline{\psi}_N:(\R^{d_o}\times \covariatespace)^{N}\to\bb{R}} ~ \sup_{h\in H : \norm{h}_{L^2_{\covariatelaw}}\leq 1/\sqrt{N}} N\bb{E}_{P^N_{\cpm_0+h}}[(\overline{\psi}_N-\Psi(\cpm_0+h))^2]\geq \norm{\widehat{\psi}_{\cpm_0}}^{2}_{L^2_{\covariatelaw}}.
     \end{equation*}
    \end{theorem}
    \begin{proof}[Proof of \Cref{thm_local_asymptotic_minimax_bound_linear_functionals}]
     We showed in the proof of \Cref{thm_local_asymptotic_normality} that the hypotheses of \cite[Theorem 3.1.3]{Nickl2023} are satisfied.
     Since $d_o\geq d_p=2\mfrak{d}$, we conclude from statement \ref{item_derivative_forward_operator_injectivity} of \Cref{prop_derivative_forward_operator} that $\bb{I}_{\cpm_0}:H\to L^2_{\covariatelaw}$ is injective. 
     Thus, we may apply \cite[Theorem 3.1.4]{Nickl2023} to obtain the desired conclusion.
    \end{proof}

    In preparation for a Bernstein--von Mises type result, we state sufficient conditions for posterior contraction with respect to the supremum norm.
      \begin{proposition}[Posterior contraction in $L^\infty$]
      \label{prop_posterior_contraction_supremum_norm}
      Suppose the hypotheses of \Cref{thm_posterior_contraction_L2_norm} hold for a regularisation space $\regspace$ that is continuously embedded into $H^{\beta}(\covariatespace)$ for $\beta> d_\covariate /2$.
      For fixed $M>0$, define
      \begin{subequations}
      \begin{align}
       \overline{\delta}_N(\beta')\coloneqq & \delta_N^{(\beta-\beta')/\beta},\quad  d_\covariate /2<\beta'<\beta,
       \label{eq_overline_delta_N}
       \\
       \cpmspace_{N,M,\infty}(\beta')\coloneqq& \{\cpm\in\cpmspace\ :\ \norm{\cpm}_{\regspace}\leq M,\ \norm{\cpm-\cpm_0}_\infty\leq M\overline{\delta}_N(\beta')\}.
       \label{eq_Theta_N_M_infty_regularisation_sets}
      \end{align}
      \end{subequations}
      Then for every $b>0$ we can choose $M$ large enough such that 
      \begin{equation*}
       P^N_{\cpm_0}\left(\Pi_N\left(\cpm\in\cpmspace_{N,M,\infty}(\beta')\middle\vert (Y^{(k)},X^{(k)})_{k\in[N]}\right)\leq 1-e^{-b N \delta_N^2}\right)\xrightarrow[N\to\infty]{}0.
      \end{equation*}
     \end{proposition}
     \begin{proof}[Proof of \Cref{prop_posterior_contraction_supremum_norm}]
      In the proof of \Cref{thm_posterior_contraction_L2_norm}, we showed that the hypotheses of \Cref{thm_posterior_contraction_L2_norm} imply that the hypotheses of \cite[Theorem 2.2.2]{Nickl2023} and \cite[Condition 2.1.4]{Nickl2023} are satisfied. Therefore, given the hypotheses on $\regspace$ and $\beta$, the desired conclusion follows from \cite[Proposition 4.1.3]{Nickl2023}.
     \end{proof}
        
    The last result of this section is a Bernstein--von Mises theorem for a class of linear functionals.
    Recall that \Cref{prop_solvability_information_equation} gives conditions such that for every $\psi\in\cpmspace$, there exists $\overline{\psi}_{\cpm_0}\in\cpmspace$ such that $\informoperator[\overline{\psi}_{\cpm_0}]=\psi$ on $\covariatespace$.
    For every $N\in\N$ and $(Y^{(k)},X^{(k)})_{k\in[N]}$, define 
     \begin{equation*}
      \widehat{\Psi}_N\coloneqq \Ang{\psi,\cpm_0}_{L^2_{\covariatelaw}}+\frac{1}{N}\sum_{k=1}^{N}\Ang{\bb{I}_{\cpm_0}[\overline{\psi}_{\cpm_0}](X^{(k)}),\varepsilon^{(k)}}_{\R^{d_o}}.
     \end{equation*}
     \begin{theorem}
     \label{thm_Bernstein_von_Mises_for_linear_functionals}
     Suppose that the hypotheses of \Cref{thm_local_asymptotic_normality} and \Cref{asmp_coefficient_map_C2} hold, and suppose the hypotheses of \Cref{prop_posterior_contraction_supremum_norm} hold with $\beta>2d_\covariate$ and some
     \begin{equation}
      \label{eq_conditions_on_exponents}
      \frac{d_\covariate}{2}<\beta'<\frac{\beta(\alpha-d_\covariate)}{3\alpha}.
     \end{equation}
      Then for a random variable $\cpm$ drawn from the posterior $\Pi_N(\cdot~\vert (Y^{(k)},X^{(k)})_{k\in[N]})$, 
     \begin{equation*}
      \sqrt{N}\bigr(\ang{\cpm,\psi}_{L^2_{\covariatelaw}}-\widehat{\Psi}_N\bigr)\bigr\vert (Y^{(k)},X^{(k)})_{k\in[N]}\xrightarrow[N\to\infty]{d} \normaldist{0}{\norm{\bb{I}_{\cpm_0}[\overline{\psi}_{\cpm_0}]}^{2}_{L^2_{\covariatelaw}}}.
     \end{equation*} 
    \end{theorem}
    The conditions \eqref{eq_conditions_on_exponents} are feasible if and only if $d_\covariate<\tfrac{2\alpha\beta}{3\alpha+2\beta}$.
    Given $d_\covariate\in\N$, setting $(\alpha,\beta)\leftarrow (4d_\covariate,3d_\covariate)$ ensures that this inequality is satisfied, for example.
    \begin{proof}[Proof of \Cref{thm_Bernstein_von_Mises_for_linear_functionals}]
    \Cref{thm_Bernstein_von_Mises_for_linear_functionals} follows from \cite[Theorem 4.1.5]{Nickl2023}.
    We verify the hypotheses of this theorem on the forward operator $\fwdmodel$ and on $\psi$.
    
    The hypotheses of \Cref{prop_posterior_contraction_supremum_norm} include the hypotheses of \Cref{thm_posterior_contraction_L2_norm}.
    In the proof of \Cref{thm_posterior_contraction_L2_norm}, we showed that the forward operator $\fwdmodel$ satisfies \cite[Conditions 2.1.1, 2.1.4]{Nickl2023}.
    In the proof of \Cref{thm_local_asymptotic_normality}, we showed that \cite[Condition 3.1.1]{Nickl2023} holds with $H=\cpmspace=L^\infty(\covariatespace,\paramspace)$.
    Since \Cref{thm_local_asymptotic_normality} assumes that the hypotheses of \Cref{prop_derivative_forward_operator_Linfty} holds, it follows that \cite[Condition 4.1.1]{Nickl2023} holds.
    Thus the hypotheses of \cite[Theorem 4.1.5]{Nickl2023} on $\fwdmodel$ are satisfied.
    
    The hypotheses of \Cref{thm_local_asymptotic_normality} include those of \Cref{prop_solvability_information_equation}.
    Thus, $\psi$ satisfies \cite[Condition 4.1.2]{Nickl2023}.
    We must show that the unique $\overline{\psi}_{\cpm_0}$ given in \Cref{prop_solvability_information_equation} satisfies \cite[Condition 4.1.4]{Nickl2023}.
    Recall the hypothesis that \Cref{asmp_coefficient_map_C2} holds.
    Since $\cpm_0\in\cpmspace=L^\infty(\covariatespace,\paramspace)$, the set $K\coloneqq\closure{\{\cpm_0(\covariate):\covariate\in\covariatespace\}}$ satisfies \Cref{asmp_image_of_cpm0_contained_in_compact_subset}.
    Thus, by statement \ref{item_derivative_forward_operator_asymptotics_quadratic} of \Cref{prop_derivative_forward_operator}, the approximation in \cite[Condition 3.1.1]{Nickl2023} is quadratic, i.e. $\rho_{\cpm_0}[h]=O(\norm{h}_\infty^2)$. 
    By \eqref{eq_delta_N_rate} and \eqref{eq_overline_delta_N},
    \begin{equation*}
     N\overline{\delta}_N^3=N^{\tfrac{\beta(2\alpha+d_\covariate)-3\alpha(\beta-\beta')}{\beta(2\alpha+d_\covariate)}}.
    \end{equation*}
    Now $\beta(2\alpha+d_\covariate)-3\alpha(\beta-\beta')<0$ if and only if $\beta'<\tfrac{\beta(\alpha-d_\covariate)}{3\alpha}$, so \eqref{eq_conditions_on_exponents} implies that $N\overline{\delta}_N^3\to 0$ as $N\to\infty$. 
    Given the hypothesis $\beta>2d_\covariate$, it follows from the discussion on \cite[p. 70]{Nickl2023} below \cite[Equation (4.6)]{Nickl2023} that \cite[Condition 4.1.4]{Nickl2023} holds.
    This completes the verification of the hypotheses of \cite[Theorem 4.1.5]{Nickl2023}.
    \end{proof}
    The results of this section follow by combining our analysis of the forward operator $\fwdmodel$ in \Cref{sec_forward_operator_fixed_time_random_covariate_design} with the framework presented in \cite{Nickl2023}.
    While this framework is motivated by PDE-based inverse problems, the inverse problems we have considered are all defined by linear time-homogeneous ODE-IVPs with matrices that are diagonalisable over $\R$.
    Thus, our results provide further evidence for the versatility of this framework.
        
    \section{Application to two-compartment model}
    \label{sec_application_two_compartment_model}
    
    In this section, we state an example of an ODE-IVP from the pharmacokinetics literature and show that the key assumptions on the forward operator $\fwdmodel$ from \Cref{sec_forward_operator_fixed_time_random_covariate_design} are satisfied.
    Thus, the analysis of the posterior from \Cref{sec_properties_of_Bayesian_inverse_problem} applies to the corresponding Bayesian inverse problem associated to the observation model \eqref{eq_fixed_time_random_covariate_design_observation_model}.
    
    Recall the two-compartment model from \eqref{eq_ODE_IVP_two_compartment_model},
    \begin{equation*}
     \begin{aligned}
    V_1 \frac{\rd }{\rd t} s_1 (t)=&Q(s_2(t)-s_1(t))-CL\ s_1(t),& V_1 s_1(0)&=D_0 w,
    \\ 
    V_2\frac{\rd}{\rd t} s_2(t)=&Q(s_1(t)-s_2(t)),& s_2(0)&=0,
    \end{aligned}
    \end{equation*}
    where $D_0$ is a reference constant and $w$ is the weight of an individual.
    The parameters of the two-compartment model above are $(CL,V_1,Q,V_2)\in\R^4_{>0}$.
    On \cite[p. 569]{Hartung2021}, the ODE-IVP is converted to another ODE-IVP, using a weight-dependent rescaling procedure described in \cite{Robbie2012}.
    We summarise this procedure below and refer interested readers to \cite{Robbie2012} for the pharmacological motivation for this procedure.
    
    Given a fixed reference weight $w_0$, define a rescaled weight $\mfrak{w}\coloneqq w/w_0$, and define the weight-normalised parameters $(CL^\ast,V_1^\ast,Q^\ast,V_2^\ast)$ according to 
    \begin{equation*}
     CL^\ast\coloneqq CL \mfrak{w}^{-3/4},\quad V_1^\ast\coloneqq V_1\mfrak{w}^{-1},\quad Q^\ast=Q \mfrak{w}^{-3/4},\quad V_2^\ast=V_2 \mfrak{w}^{-1}.
    \end{equation*}
    Since $(CL,V_1,Q,V_2)\in\R^4_{>0}$ and $w>0$ for every individual, dividing the first and second row of \eqref{eq_ODE_IVP_two_compartment_model} by $V_1$ and $V_2$ respectively and using the definitions of $CL^\ast$, $V_1^\ast$, $Q^\ast$, and $V_2^\ast$ yields
    \begin{equation}
    \label{eq_ODE_IVP_two_compartment_model_rescaled}
     \begin{aligned}
    \frac{\rd }{\rd t} s_1 (t)=&\mfrak{w}^{-1/4}\left( \frac{Q^\ast}{V_1^\ast}(s_2(t)-s_1(t))-\frac{CL^\ast}{V_1^\ast} s_1(t)\right),& s_1(0)&=\frac{D_0 w_0}{V_1^\ast},
    \\ 
    \frac{\rd}{\rd t} s_2(t)=& \mfrak{w}^{-1/4}\frac{Q^\ast}{V_2^\ast}(s_1(t)-s_2(t)),& s_2(0)&=0.
    \end{aligned}
    \end{equation}
    In \cite{Hartung2021}, the covariate consists of the rescaled weight and the age, i.e. $\covariate=(\mfrak{w},a)$, and we define $\covariatespace$ to be a bounded subset of $\R^2_{>0}$ with smooth boundary, such that $\inf\{\covariate_i:\covariate\in\covariatespace\}>0$ for $i=1,2$.
    This assumption ensures that $\log \mfrak{w}^{-1/4}$ is bounded from above and from below, and is reasonable given that medical treatments can be given only within prescribed age and weight ranges.
    The dimension of $\covariatespace$ satisfies $d_\covariate=2$.
    
    We now choose the exponential function as a link function and define 
    \begin{equation*}
     p_1\coloneqq \log(\mfrak{w}^{-1/4}CL^\ast),\quad p_2\coloneqq \log V_1^\ast,\quad p_3\coloneqq \log (\mfrak{w}^{-1/4} Q^\ast),\quad p_4\coloneqq \log V_2^\ast,
    \end{equation*}
    so that the dimension $d_p$ of the parameter space $\paramspace$ satisfies $d_p=4$. 
    Then \eqref{eq_ODE_IVP_two_compartment_model_rescaled} becomes
    \begin{equation}
    \label{eq_ODE_IVP_two_compartment_model_rescaled_reparametrised}
    \frac{\rd}{\rd t}\begin{bmatrix}
                      s_1(t,p)
                      \\
                      s_2(t,p)
                     \end{bmatrix}=
     \begin{bmatrix}
-e^{p_1-p_2}-e^{p_3-p_2} & e^{p_3-p_2}
\\
e^{p_3-p_4} & -e^{p_3-p_4}
\end{bmatrix}\begin{bmatrix}
                      s_1(t,p)
                      \\
                      s_2(t,p)
                     \end{bmatrix}
                     ,\quad \begin{bmatrix}
                      s_1(0,p)
                      \\
                      s_2(0,p)
                     \end{bmatrix}=\begin{bmatrix} 
 D_0 w_0 e^{-p_2}
 \\
 0
\end{bmatrix},
\end{equation}
which is of the form \eqref{eq_ODE_IVP}.

\begin{remark}
On \cite[p. 569]{Hartung2021}, only $CL^\ast$ is given as a function of age $a$, while $V_1^\ast$, $Q^\ast$, and $V_2^\ast$ are unknown constants. 
Our analysis of the ODE-IVP \eqref{eq_ODE_IVP_two_compartment_model_rescaled_reparametrised} applies in the more general setting where one allows $CL^\ast$, $V_1^\ast$, $Q^\ast$, and $V_2^\ast$ to all depend on $a$ and $\mfrak{w}$. 
\end{remark}
\begin{remark}
The weight normalisation procedure described above could be omitted to allow for more flexibility in how the parameters $(p_i)_{i\in[4]}$ depend on $\covariate=(\mfrak{w},a)$. By dividing the first and second rows of \eqref{eq_ODE_IVP_two_compartment_model} by $V_1$ and $V_2$ respectively and by defining $p_1\coloneqq \log CL$, $p_2\coloneqq \log V_1$, $p_3\coloneqq \log  Q$ and $p_4\coloneqq \log V_2$, one obtains an ODE-IVP that differs from \eqref{eq_ODE_IVP_two_compartment_model_rescaled_reparametrised} only in the value of $s_1(0,p)$, namely that $s_1(0,p)=D_0 w_0 e^{-p_2}$ in \eqref{eq_ODE_IVP_two_compartment_model_rescaled_reparametrised} is replaced with $s_1(0,p)= D_0 w e^{-p_2}$.
    \end{remark}
    
    Recall \Cref{def_intrinsic_dimension_of_ODE_IVP_solution_first_component} of the intrinsic dimension $2\mfrak{d}$.
    The next result shows that the ODE-IVP \eqref{eq_ODE_IVP_two_compartment_model_rescaled_reparametrised} satisfies the key assumptions from \Cref{sec_properties_of_solutions_to_linear_ODE_IVPs}.    
 \begin{proposition}
 \label{prop_two_compartment_model_satisfies_all_assumptions}
The ODE-IVP \eqref{eq_ODE_IVP_two_compartment_model_rescaled_reparametrised} satisfies \Cref{asmp_local_boundedness_and_local_lipschitz_continuity_of_matrix_and_IC}, \Cref{asmp_ODE_matrix_diagonalisable}, \Cref{asmp_coefficient_map_Jacobian_full_rank}, \Cref{asmp_applicability_of_observation_function}, and \Cref{asmp_coefficient_map_C2} for $\paramspace\leftarrow \R^4$, and the intrinsic dimension satisfies $2\mfrak{d}=4$.
If $\cpmspace\coloneqq L^\infty(\covariatespace,\paramspace)$, then every $\cpm_0\in\cpmspace$ satisfies \Cref{asmp_image_of_cpm0_contained_in_compact_subset}.
 \end{proposition}
 \begin{proof}
  In \Cref{lem_two_compartment_model_hypotheses_for_contraction_in_dG_metric}, we show that the ODE-IVP \eqref{eq_ODE_IVP_two_compartment_model_rescaled_reparametrised} satisfies \Cref{asmp_local_boundedness_and_local_lipschitz_continuity_of_matrix_and_IC}.
  In \eqref{eq_two_compartment_model_diagonalisable}, we recall the diagonalisation given in the supplemental material of \cite{Hartung2021}; this shows that  \eqref{eq_ODE_IVP_two_compartment_model_rescaled_reparametrised} satisfies \Cref{asmp_ODE_matrix_diagonalisable}.
  In \Cref{lem_two_compartment_model_hypotheses_for_stability}, we show that the intrinsic dimension satisfies $2\mfrak{d}=4$, and that \Cref{asmp_coefficient_map_Jacobian_full_rank}, \Cref{asmp_applicability_of_observation_function}, and \Cref{asmp_coefficient_map_C2} hold.
In the proof of \Cref{prop_derivative_forward_operator_Linfty}, we showed that if $\cpm_0\in L^\infty(\covariatespace,\paramspace)$, then the set $K\coloneqq \closure{\{\cpm_0(\covariate):\covariate\in\covariatespace\}}$ satisfies \Cref{asmp_image_of_cpm0_contained_in_compact_subset}.
 \end{proof}
We end this section with the following result, which establishes posterior contraction and a Bernstein--von Mises result for the two-compartment model \eqref{eq_ODE_IVP_two_compartment_model_rescaled_reparametrised} from \cite{Hartung2021,Robbie2012}.
 \begin{theorem}
  \label{thm_posterior_contraction_Linfty_for_two_compartment_model}
  Let $d_o\geq 4$ and $\cpmspace=H=L^\infty(\covariatespace,\paramspace)$. Suppose that $\Pi'$ satisfies \Cref{asmp_Gaussian_base_prior} for $\alpha>0$ and $\regspace$ given in \Cref{prop_posterior_contraction_supremum_norm} for $\beta>1$, and suppose $\cpm_0\in \mcal{H}\cap \regspace$.
  Then for $\overline{\delta}_N(\beta')$ and $\Theta_{N,M,\infty}(\beta')$ given in \eqref{eq_overline_delta_N} and \eqref{eq_Theta_N_M_infty_regularisation_sets} respectively, the conclusion of \Cref{prop_posterior_contraction_supremum_norm} holds.
  If in addition $\beta>4$, then the conclusion of \Cref{thm_Bernstein_von_Mises_for_linear_functionals} holds.
 \end{theorem}
 \begin{proof}[Proof of \Cref{thm_posterior_contraction_Linfty_for_two_compartment_model}]
  By \Cref{prop_two_compartment_model_satisfies_all_assumptions} and the hypotheses that $d_o\geq 4=d_p=2\mfrak{d}$, the hypotheses of \Cref{prop_stability_estimate} and \Cref{prop_derivative_forward_operator_Linfty} hold.
  Since $\beta>1=d_\covariate/2$, the hypotheses of \Cref{thm_posterior_contraction_L2_norm}, \Cref{thm_local_asymptotic_normality}, and \Cref{prop_posterior_contraction_supremum_norm} also hold.
  The first conclusion then follows from \Cref{prop_posterior_contraction_supremum_norm}.
  If in addition $\beta>4=2d_\covariate$, then the hypotheses of \Cref{thm_Bernstein_von_Mises_for_linear_functionals} holds, and the second conclusion follows.
 \end{proof}

\section{Conclusion}
\label{sec_conclusion}
 
  We considered parametrised linear homogeneous ODE-IVPs defined by a matrix $A(p)$ and initial condition $s(0,p)=s_0(p)$ on a fixed finite time interval. Given observations of only one component of the solution at a fixed finite collection of observation times $(t_j)_{j\in [d_o]}$, and given a priori chosen sets $\covariatespace$ and $\paramspace$ of admissible covariates and parameters, the task was to infer a covariate-to-parameter map $\cpm_0$ that maps the covariate vector $x$ for every individual in a population to the corresponding parameter vector $p=\cpm_0(x)$ that defines the ODE-IVP model for that individual.
 The constraint that only one component of the solution is observed is motivated by pharmacological settings, where only drug concentrations in the patient's blood can be measured.
 
 In this work, we presented a framework for proving posterior contraction, local asymptotic normality, and Bernstein--von Mises theorems for the Bayesian inverse problem of inferring the covariate-to-parameter map $\cpm_0$ using pairs $(Y^{(k)},X^{(k)})_{k\in\N}$ of covariates and observation vectors, where each $Y^{(k)}$ is related to $X^{(k)}$ by the fixed finite time design and random covariate design model 
 \begin{equation*}
  Y^{(k)}=(\log s_1(t_j,\cpm_0(X^{(k)})))_{j\in[d_o]}+\noise^{(k)},\quad \noise^{(k)}\distiid \normaldist{0}{\varsigma^2 I},\ X^{(k)}\distiid \covariatelaw,
 \end{equation*}
 and the true population distribution $\covariatelaw$ of the $(X^{(k)})_{k\in\N}$ is unknown.

 Our framework involves some key assumptions. First, in \Cref{asmp_local_boundedness_and_local_lipschitz_continuity_of_matrix_and_IC}, the ODE right-hand side $A(p)$ and initial condition $s_0(p)$ are assumed to be locally bounded and locally Lipschitz continuous functions of the parameter $p$.
 We then further assume in \Cref{asmp_ODE_matrix_diagonalisable} that $A(p)$ is diagonalisable in $\R$.
 The purpose of \Cref{asmp_ODE_matrix_diagonalisable} is to ensure that the observed component of the solution can be expressed as a sum of exponential functions \eqref{eq_ODE_IVP_solution_first_component_exponential_form}.
 As a sum of exponential functions, the observed component is now completely determined by the so-called coefficient map introduced in \Cref{def_intrinsic_dimension_of_ODE_IVP_solution_first_component} and the intrinsic dimension $2\mfrak{d}$.
 In addition, being a sum of exponentials facilitates the analysis of the Jacobian of the evaluation map $p\mapsto (s_1(t_j,p))_{j\in[d_o]}$ in terms of the Jacobian of the coefficient map.
 If \Cref{asmp_coefficient_map_Jacobian_full_rank} holds, i.e. if the  Jacobian of the coefficient map exists and satisfies a full rank condition, then we can exploit the structure of being a sum of exponentials to show that the evaluation map has a locally Lipschitz inverse; see \Cref{lem_locally_Lipschitz_inverse_evaluation_map}.
 Using the strategy of proving estimates that are uniform in the covariate, our framework avoids imposing any constraints on the unknown covariate law $\covariatelaw$.

 Our framework is designed to be applied together with the framework presented in \cite{Nickl2023} for nonlinear Bayesian inverse problems.
 In our framework, \Cref{lem_locally_Lipschitz_inverse_evaluation_map} is the most important result, because it leads directly to the  stability estimate that is essential in \cite{Nickl2023} for obtaining posterior contraction rates for the unknown covariate-to-parameter map $\cpm_0$ from contraction rates for $\fwdmodel(\cpm_0)$.
 In addition, \Cref{lem_locally_Lipschitz_inverse_evaluation_map} shows that in the fixed finite time design setting, the lower bound $d_o\geq 2\mfrak{d}$ on the size $d_o$ of the time design $(t_j)_{j\in[d_o]}$ is a sufficient condition for stability, and thereby for posterior contraction for $\cpm_0$. 
 The condition $d_o\geq 2\mfrak{d}$ is also important for showing that the derivative of the forward operator is injective and for showing that the information equation is solvable; see statement \ref{item_derivative_forward_operator_injectivity} of \Cref{prop_derivative_forward_operator} and \Cref{prop_solvability_information_equation} respectively.
 In the context of Bayesian inference, the key results of this work are the results in \Cref{sec_properties_of_Bayesian_inverse_problem}, e.g. the posterior contraction and Bernstein--von Mises results in \Cref{thm_posterior_contraction_L2_norm} and \Cref{thm_Bernstein_von_Mises_for_linear_functionals} respectively.
 
 For future work, it would be interesting to investigate whether one can reduce further the number $d_o$ of observations taken for each covariate drawn from the population. This would be relevant for applications in pharmacology, where there are some groups of individuals for which clinicians may only be able to draw one blood sample per individual. One may also investigate whether the diagonalisability property in \Cref{asmp_ODE_matrix_diagonalisable} can be weakened, and to consider more general families of ODEs. 
 
  \section*{Acknowledgements}

The research of the author has been partially funded by the Deutsche Forschungsgemeinschaft (DFG) --- Project-ID \href{https://gepris.dfg.de/gepris/projekt/318763901}{318763901} --- SFB1294.
The author would like to thank Giuseppe Carere, Niklas Hartung, and Wilhelm Huisinga for helpful discussions.
 
\appendix

\section{Auxiliary results}
\label{appendix_auxiliary_results}

We use the following local Lipschitz continuity bounds:
\begin{subequations}
 \begin{align}
\label{eq_local_Lipschitz_continuity_logarithm}
 \forall t,s>0,\quad \abs{\log t-\log s}\leq& \frac{\abs{t-s}}{s\wedge t},
 \\
 \label{eq_local_Lipschitz_continuity_exponential}
 \forall t,s\in\R,\quad \abs{\exp t-\exp s} \leq& \exp(s\vee t)\abs{t-s}.
\end{align}
\end{subequations}

    We state and prove a result that is inspired but not implied by the main result of \cite{Tossavainen2007}.
    We use this result in \Cref{lem_locally_Lipschitz_inverse_evaluation_map}.
     \NumRootsLinCombExpAffProd*
    \begin{proof}[Proof of \Cref{lem_number_of_roots_of_linear_combination_of_exponential_affine_products}]

 We prove the claim by induction. Let $n=1$. Consider the function $\R\ni t\mapsto e^{\beta t}(\xi+\gamma t)$. 
 If $\gamma=0$, then $\xi\neq 0$, and the function has no roots.
 If $\gamma\neq 0$, then the function has exactly one root at $t=-\xi/\gamma$.
 This proves the statement in the base case.
 
 Now suppose that the claim is true for some $n\leq 1$.
 Since
 \begin{equation*}
  h(t)\coloneqq\sum_{k\in[n+1]}e^{\beta_k t}(\xi_k+\gamma_k t)=e^{\beta_1 t}\left[ (\xi_1+\gamma_1t)+\sum_{k=2}^{n+1} e^{(\beta_k-\beta_1)t}(\xi_k+\gamma_kt)\right]\eqqcolon e^{\beta_1 t}\widehat{h}(t),
 \end{equation*}
 it follows that $h(t)=0$ if and only if $\widehat{h}(t)=0$.
 Note that for arbitrary $\beta,\gamma,\xi\in\R$,
 \begin{equation*}
  \frac{\rd^2}{\rd t^2} e^{\beta t}(\xi+\gamma t)=\frac{\rd}{\rd t} \left[ e^{\beta t}\left(\beta(\xi+\gamma t)+\gamma\right)\right]=e^{\beta t}\left[ \beta^2\xi+2\beta\gamma+\beta\gamma t\right].
 \end{equation*}
 Hence, 
 \begin{align*}
  &\frac{\rd^2 }{\rd t^2}\widehat{h}(t)= \frac{\rd^2}{\rd t^2} \sum_{k=2}^{n+1} e^{(\beta_k-\beta_1)t}(\xi_k+\gamma_kt)
  \\
  =& \sum_{k=2}^{n+1} e^{(\beta_k-\beta_1)t}\left[(\beta_k-\beta_1)^2\xi_k+2(\beta_k-\beta_1)\gamma_k+(\beta_k-\beta_1)\gamma_k t\right].
 \end{align*}
 By the hypotheses, the $(\beta_k-\beta_1)_{k=2}^{n+1}$ are distinct, and $((\beta_{k}-\beta_1)\gamma_{k},(\beta_k-\beta_1)^2\xi_k+2(\beta_k-\beta_1)\gamma_k)_{k=2}^{n+1}\in \R^{2n}\setminus\{0\}$.
 Thus we may apply the base case to conclude that $\tfrac{\rd^2 }{\rd t^2}\widehat{h}(t)$ has at most $2n-1$ roots.
 By the mean value theorem, $\tfrac{\rd}{\rd t}\widehat{h}(t)$ has at most $2n$ roots, and thus $\widehat{h}(t)$ has at most $2n+1$ roots. 
 Since $h(t)=e^{\beta_1 t}\widehat{h}(t)$, it follows that $h$ has at most $2n+1=2(n+1)-1$ roots.
 This proves the inductive step and completes the proof of \Cref{lem_number_of_roots_of_linear_combination_of_exponential_affine_products_function}.
\end{proof}

\section{Properties of the two-compartment model}
\label{sec_properties_of_two_compartment_model}

Let $\kappa\coloneqq D_0 w_0>0$. Then \eqref{eq_ODE_IVP_two_compartment_model_rescaled_reparametrised} is of the form \eqref{eq_ODE_IVP}, for $A$ and $s_0$ given in \eqref{eq_matrix_and_IC_param_reformulation} below:
\begin{subequations}
\label{eq_param_reformulation}
\begin{align}
A(p)\coloneqq &
\begin{bmatrix}
-e^{p_3-p_2}-e^{p_1-p_2} & e^{p_3-p_2}
\\
e^{p_3-p_4} & -e^{p_3-p_4}
\end{bmatrix},& &
s_0(p)\coloneqq 
\begin{bmatrix} 
 \kappa e^{-p_2}
 \\
 0
\end{bmatrix},
\label{eq_matrix_and_IC_param_reformulation}
\\
\sigma(p)\coloneqq & e^{p_3-p_2}+e^{p_1-p_2}+e^{p_3-p_4}, & &\delta(p)\coloneqq \sqrt{\sigma(p)^2-4e^{p_1-p_2} e^{p_3-p_4}},
\label{eq_sigma_delta_param_reformulation}
\\
\lambda^{\pm}(p)\coloneqq& \frac{\pm \delta(p) -\sigma(p)}{2},& & v^{\pm}(p)\coloneqq \lambda^{\pm}(p) e^{p_4-p_3}+1,
 \label{eq_eigenpairs_param_reformulation}
 \\
v^+(p)-v^-(p)=&\delta(p) e^{p_4-p_3}, & & \eta(p)\coloneqq \frac{\kappa }{e^{p_2}(v^{+}(p)-v^{-}(p))}.
 \label{eq_eta_param_reformulation}
 \end{align}
\end{subequations}
Since $A(p)\in \R^{2\times 2}$, we have $d_s=2$.

We now show that
  \begin{equation}
\label{eq_eigenvalue_eigenvector_inequalities}
 \delta(p)>0,\quad \lambda^-(p)<\lambda^+(p)<0,\quad  v^-(p)<0<v^+(p).
\end{equation}
  For $\alpha,\beta,\gamma\in\R$,
  \begin{equation*}
   (\beta+\gamma)^2 - 4\beta\gamma =\beta^2 -2\beta\gamma +\gamma^2 \geq 0 
  \end{equation*}
  and 
  \begin{align*}
   &(\alpha+\beta+\gamma)^2-4 \beta \gamma -(\alpha+\beta-\gamma)^2
   \\
   =& (\alpha+\beta+\gamma +\alpha+\beta-\gamma)(\alpha+\beta+\gamma-(\alpha+\beta -\gamma))-4\beta\gamma
   \\
   =& 4(\alpha+\beta)\gamma-4\beta\gamma=4\alpha\gamma.
  \end{align*}
  Hence, if $\alpha,\beta,\gamma>0$,
   \begin{equation*}
(\alpha+\beta+\gamma)^2> (\beta+\gamma)^2\geq 4\beta\gamma,\quad \sqrt{(\alpha+\beta+\gamma)^2-4\beta\gamma}>\abs{\alpha+\beta-\gamma}.
\end{equation*}
Set $\alpha\leftarrow e^{p_3-p_3}$, $\beta\leftarrow e^{p_1-p_2}$, and $\gamma\leftarrow e^{p_3-p_4}$. 
Then by \eqref{eq_sigma_delta_param_reformulation}, $\sigma(p)=\alpha+\beta+\gamma$ and $\delta(p)=\sqrt{(\alpha+\beta+\gamma)^2-4\beta\gamma}$.
Thus, the inequality $(\alpha+\beta+\gamma)^2>4\beta\gamma$ is equivalent to $\delta>0$, while $\delta<\sigma$ holds by definition of $\delta(p)$.
By the definitions \eqref{eq_eigenpairs_param_reformulation} of $\lambda^{\pm}(p)$ and $v^{\pm}(p)$, it follows that $\lambda^{-}(p)<\lambda^{+}(p)<0$.
Since
\begin{equation*}
 v^{\pm}(p)=\frac{\pm\sqrt{(\alpha+\beta+\gamma)^2-4\beta\gamma}-(\alpha+\beta+\gamma)+2\gamma}{2\gamma}=\frac{\pm\sqrt{(\alpha+\beta+\gamma)^2-4\beta\gamma}-(\alpha+\beta-\gamma)}{2\gamma}
\end{equation*}
and $\sqrt{(\alpha+\beta+\gamma)^2-4\beta\gamma}>\abs{\alpha+\beta-\gamma}$, it follows that $v^{+}(p)>0>v^{-}(p)$.

We now use the eigenvalue inequalities in \eqref{eq_eigenvalue_eigenvector_inequalities} to prove local boundedness and local Lipschitz continuity of solutions to \eqref{eq_ODE_IVP} for $A(p)$ and $s_0(p)$ given in \eqref{eq_matrix_and_IC_param_reformulation}.

\begin{restatable}{lemma}{twoCompModelHyposForContractionIndG}
   \label{lem_two_compartment_model_hypotheses_for_contraction_in_dG_metric}
   Let $A(p)$ and $s_0(p)$ be as in \eqref{eq_matrix_and_IC_param_reformulation}. For every $M>0$, 
   \begin{subequations}
   \begin{align}
    \sup\{\norm{A(p)}_2\vee \norm{s_0(p)}_2:p\in B_2(0,M) \}\leq& (3\vee \kappa)e^{2M},
    \label{eq_two_compartment_model_local_boundedness_hypothesis}
    \\
    \sup\left\{\norm{A(p)-A(q)}_2\vee \norm{s_0(p)-s_0(q)}_2:p,q\in B_2(0,M) \right\}\lesssim & (6\vee \kappa) e^{2M}\norm{p-q}_2.
    \label{eq_two_compartment_model_local_Lipschitz_hypothesis}
    \end{align}
   \end{subequations}
   In particular, \Cref{asmp_local_boundedness_and_local_lipschitz_continuity_of_matrix_and_IC} holds, with $C_1(M)\lesssim (6\vee \kappa)e^{2M}$.
   \end{restatable} 
   \begin{proof}[Proof of \Cref{lem_two_compartment_model_hypotheses_for_contraction_in_dG_metric}]
    Since $\norm{A(p)}_2$ denotes the spectral norm of $A(p)$, we have
    \begin{equation*}
     \norm{A(p)}=\abs{\lambda^{-}(p)}=\tfrac{1}{2}\abs{\delta(p)+\sigma(p)}\leq \tfrac{1}{2}\abs{\sqrt{\sigma(p)^2}+\sigma(p)}=\sigma(p)\leq 3e^{2M},
    \end{equation*}
     where the first equation follows from the eigenvalue inequalities in \eqref{eq_eigenvalue_eigenvector_inequalities}, the second equation follows from the definition of $\lambda^{+}(p)$ in \eqref{eq_eigenpairs_param_reformulation}, the first inequality follows from the definition of $\delta(p)$ in \eqref{eq_sigma_delta_param_reformulation}, and the last inequality follows from the definition of $\sigma(p)$ in \eqref{eq_sigma_delta_param_reformulation} combined with the fact that $p\in  B_2(0,M) $ implies that $\norm{p}_\infty< M$.
     By definition of $s_0(p)$ in \eqref{eq_matrix_and_IC_param_reformulation}, $\norm{s_0(p)}_2\leq \kappa e^M$, for $p\in B_2(0,M) $. This proves \eqref{eq_two_compartment_model_local_boundedness_hypothesis}.
     
     Since $A\in\R^{d_s\times d_s}$, we have $\norm{A(p)}_2\lesssim \norm{A(p)}_\infty$, where the hidden constant depends only on $d_s$. 
     Recalling that $\norm{A(p)}_\infty$ is the maximum absolute row sum norm of $A(p)$ and using the definition \eqref{eq_matrix_and_IC_param_reformulation} of $A(p)$, we obtain
     \begin{align*}
      \norm{A(p)-A(q)}_\infty\leq \max\left\{ 2\Abs{e^{p_3-p_2}-e^{q_3-q_2}}+\Abs{e^{p_1-p_2}-e^{q_1-q_2}},2\Abs{e^{p_3-p_4}-e^{q_3-q_4}}\right\}.
     \end{align*}
    By the local Lipschitz continuity \eqref{eq_local_Lipschitz_continuity_exponential} of the exponential function, the triangle inequality, and the fact that $p,q\in B_2(0,M) $, we have for every $i,j\in[d_p]$ with $i\neq j$ that
\begin{equation*}
 \Abs{e^{q_i-q_j}-e^{p_i-p_j}}\leq e^{\max\left\{ q_i-q_j,p_i-p_j\right\}}\Abs{ (q_i-q_j)-(p_i-p_j)}\leq  2 e^{2  M} \norm{q-p}_\infty.
\end{equation*}
Combining the preceding inequalities yields
\begin{equation*}
 \norm{A(p)-A(q)}_{\infty}\leq 6 e^{2  M} \norm{p-q}_{\infty}\leq 6e^{2M}\norm{p-q}_2.
\end{equation*}
By \eqref{eq_matrix_and_IC_param_reformulation}, \eqref{eq_local_Lipschitz_continuity_exponential}, and $p,q\in  B_2(0,M) $,
\begin{equation*}
 \norm{s_0(p)-s_0(q)}_2=\kappa\abs{e^{-p_2}-e^{-q_2}}\leq \kappa e^{\max\{-p_2,-q_2\}}\abs{p_2-q_2}\leq \kappa e^{M}\norm{p-q}_2.
\end{equation*}
This proves \eqref{eq_two_compartment_model_local_Lipschitz_hypothesis}.
   \end{proof}

      One can show by direct calculations that
\begin{equation}
\label{eq_two_compartment_model_diagonalisable}
 A(p)V(p)=V(p) \Lambda(p),\quad V(p)\coloneqq\begin{bmatrix}
      v^{+}(p) & v^{-}(p) \\ 1 & 1
     \end{bmatrix},
\quad \Lambda(p)\coloneqq \begin{bmatrix}
   \lambda^{+}(p) & 0 \\ 0 & \lambda^{-}(p)
  \end{bmatrix},
\end{equation}
see e.g. the supplemental material of \cite{Hartung2021}.
Thus, the matrix $A(p)$ in \eqref{eq_matrix_and_IC_param_reformulation} satisfies \Cref{asmp_ODE_matrix_diagonalisable}.
Since $s(t,p)=e^{A(p)t}s_0(p)$ and $AV=V\Lambda$, one can show that the solution to \eqref{eq_ODE_IVP} with $A(p)$ and $s_0(p)$ defined in \eqref{eq_matrix_and_IC_param_reformulation} satisfies
\begin{equation*}
 \begin{bmatrix}
 s_1(t,p) \\ s_2(t,p)
 \end{bmatrix}
 =\eta(p)
 \begin{bmatrix}
  v^{+}(p)e^{\lambda^{+}(p)t}-v^{-}(p)e^{\lambda^{-}(p)t}
  \\
  e^{\lambda^{+}(p)t}-e^{\lambda^{-}(p)t}.
 \end{bmatrix}
\end{equation*}
In particular, \eqref{eq_ODE_IVP_solution_first_component_exponential_form} is valid:
\begin{equation}
\label{eq_two_compartment_model_first_component_exponential_form}
s_1(t,p)=\eta v^{+}(p)e^{\lambda^{+}(p)t}-\eta v^{-}(p)e^{\lambda^{-}(p)t},\quad \forall t\geq 0.
\end{equation}

\begin{lemma}
 \label{lem_two_compartment_model_hypotheses_for_stability}
 The coefficient map $p\mapsto (\eta v^{+},-\eta v^{-},\lambda^{+},\lambda^{-})(p)\in\R^{2}_{>0}\times \R^2_{<0}$ of \eqref{eq_two_compartment_model_first_component_exponential_form} is $C^2$ smooth, the intrinsic dimension satisfies $2\mfrak{d}=4$, and 
 \begin{equation}
  \label{eq_two_compartment_model_coefficients_map_Jacobian}
  \forall q\in\paramspace,\quad \mcal{J}(q)\coloneqq 
  \begin{bmatrix}
   \nabla_p \eta v^{+}(p) \vert \nabla_p \eta v^{-}(p)\vert \nabla_p \lambda^{+}(p)\vert \nabla_p \lambda^{-}(p)
  \end{bmatrix}^\top \vert_{p=q}\in GL(4,\R).
 \end{equation}
Thus, \Cref{asmp_coefficient_map_Jacobian_full_rank}, \Cref{asmp_applicability_of_observation_function}, and  \Cref{asmp_coefficient_map_C2} hold.
\end{lemma}
\begin{proof}[Proof of \Cref{lem_two_compartment_model_hypotheses_for_stability}]
Recall \Cref{def_intrinsic_dimension_of_ODE_IVP_solution_first_component} of the intrinsic dimension $2\mfrak{d}$ in terms of the number $\mfrak{d}$ of distinct eigenvalues. By \eqref{eq_eigenvalue_eigenvector_inequalities}, $\mfrak{d}=2$, so the intrinsic dimension of the ODE-IVP \eqref{eq_ODE_IVP_two_compartment_model_rescaled_reparametrised} is $2\mfrak{d}=4$.

It follows from \eqref{eq_eigenpairs_param_reformulation} and \eqref{eq_eta_param_reformulation} that $\eta$, $v^{\pm}$, and $\lambda^{\pm}$ are $C^\infty$ functions of $p$. Thus, the coefficient map $p\mapsto (\eta v^{+},-\eta v^{-},\lambda^{+},\lambda^{-})(p)$ for $s_1$ in \eqref{eq_two_compartment_model_first_component_exponential_form} is $C^2$, which shows that \Cref{asmp_coefficient_map_C2} holds.
By \eqref{eq_eigenvalue_eigenvector_inequalities}, $p\mapsto (\eta v^{+},\eta v^{-},\lambda^{+},\lambda^{-})\in\R_{>0}^{2}\times\R^2_{<0}$.
By the eigenvalue inequalities $\lambda^{+}>\lambda^{-}$ in \eqref{eq_eigenvalue_eigenvector_inequalities}, the sum of exponentials  \eqref{eq_two_compartment_model_first_component_exponential_form}, and the definition of $\eta$ in \eqref{eq_eta_param_reformulation}, we have
\begin{equation*}
 s_1(t,p)\geq \left(\eta v^{+}(p)-\eta v^{-}(p)\right)e^{\lambda^{-}(p)t}=\kappa e^{-p_2+\lambda^{-}(p)t} >0,
\end{equation*}
so \Cref{asmp_applicability_of_observation_function} holds.

It remains to show that \Cref{asmp_coefficient_map_Jacobian_full_rank} holds. For matrices $M$ and $N$ of the same size, let $M\sim N$ denote the property that $M$ and $N$ are column equivalent, i.e. one can change $M$ to $N$ and vice versa by elementary column operations, and let $M\propto N$ denote the property that there exists some scalar $\lambda\in\R$ such that $M=\lambda N$. 
  Since $\lambda^{+}-\lambda^{-}=\delta$ and $\lambda^{+}+\lambda^{-}=-\sigma$ by \eqref{eq_eigenpairs_param_reformulation}, the matrix $\mcal{J}(q)$ in \eqref{eq_two_compartment_model_coefficients_map_Jacobian} satisfies
 \begin{equation}
  \label{eq_two_compartment_model_coefficients_map_Jacobian1}
 \mcal{J}(q)\sim 
 \begin{bmatrix}
  \nabla_p \eta v^{+}(p) \vert \nabla_p \eta v^{-}(p)\vert \nabla_p \sigma(p)\vert \nabla_p \delta(p)
  \end{bmatrix}^\top \vert_{p=q}.
 \end{equation}
   For $i\in[4]$, let $e_i$ denote the $i$-th canonical basis vector of $\R^4$.
  Now note that 
 \begin{align*}
  \nabla_p \sigma(p)=&(
  e^{p_1-p_2} , -e^{p_3-p_2}-e^{p_1-p_2}, e^{p_3-p_2}+e^{p_3-p_4},-e^{p_3-p_4})
  \\
  \nabla_p\delta(p)=& \sigma \delta^{-1}\nabla_p \sigma(p)+2\delta^{-1}(p)e^{p_1-p_2+p_3-p_4}(-e_1+e_2-e_3+e_4),
 \end{align*}
 and $\nabla_p \delta,\nabla_p\lambda^{\pm}\in\spn\{\nabla_p\sigma,-e_1+e_2-e_3+e_4\}$.
  Since the second component of
  \begin{equation*}
   -e^{-p_1+p_2} \nabla_p\sigma(p)=(-1,e^{p_3-p_2-p_1+p_2}+1,e^{p_3-p_2-p_1+p_2}-e^{p_3-p_4-p_1+p_2},e^{p_3-p_4-p_1+p_2})
  \end{equation*}
  is strictly larger than the second component of $-e_1+e_2-e_3+e_4$, it follows that $\nabla_p \sigma(p)$ and $u$ are linearly independent.
  Thus by \eqref{eq_two_compartment_model_coefficients_map_Jacobian1},
  \begin{equation*}
   \mcal{J}(q)\sim \begin{bmatrix}
  \nabla_p \eta v^{+}(p) \vert \nabla_p \eta v^{-}(p)\vert \nabla_p \sigma(p)\vert~ -e_1+e_2-e_3+e_4~
  \end{bmatrix}^\top \vert_{p=q}.
  \end{equation*}
  By \eqref{eq_eta_param_reformulation},
 \begin{equation*}
  \eta v^{\pm}(p)=\frac{\kappa}{e^{p_2+p_4-p_3}\delta(p)}\left(\lambda^{\pm}e^{p_4-p_3}+1 \right)=\frac{\kappa}{\delta}\left( \lambda^{\pm}e^{-p_2}+e^{-p_2-p_4+p_3}\right).
 \end{equation*}
 By the equation above and the product rule,
  \begin{equation*}
   \nabla_p (\eta v^{\pm})=\frac{\kappa}{\delta}\left(\frac{e^{-p_2}\lambda^{\pm}+e^{-p_2+p_3-p_4}}{\delta}(- \nabla_p\delta)+\lambda^{\pm} \nabla_p e^{-p_2}+e^{p_2}\nabla \lambda^{\pm}+\nabla_p e^{-p_2+p_3-p_4}\right).
  \end{equation*}
  Then $\nabla_p e^{-p_2}\propto  -e_2$ and $\nabla_p e^{-p_2+p_3-p_4}\propto -e_2+e_3-e_4$.
  Recalling that $\nabla_p \delta,\nabla_p\lambda^{\pm}\in\spn\{\nabla_p\sigma,u\}$ implies that $\nabla_p \eta v^{\pm}$ is a linear combination of $\nabla_p\sigma$, $u$, $v$, and $w$, which implies $\mcal{J}(q)\sim [\nabla_p \sigma \vert~ u~ \vert~ v~\vert~ w~]^\top\vert_{p=q}$.
  Let $(e_i)_{i\in[4]}$ denote the canonical orthonormal basis of $\R^4$. 
  Since
  \begin{equation*}
   \begin{bmatrix}
    ~w ~\vert~ v~\vert~ u~
   \end{bmatrix}
=
\begin{bmatrix}
 0 & 0 & -1 \\ -1 & -1 & +1 \\ +1 & 0 & -1 \\ -1 & 0 & +1
\end{bmatrix}
\sim 
\begin{bmatrix}
 0 & 0 & +1 \\ 0 & +1 & 0 \\ +1 & 0 & 0 \\ -1 & 0 & 0
\end{bmatrix}=\begin{bmatrix}
~e_3-e_4~\vert~ e_2~\vert~ e_1~
\end{bmatrix},
\end{equation*}
it follows that $\mcal{J}(q)$ has full rank if and only if $\nabla_p\sigma(p)\notin \spn\{e_3-e_4,e_2,e_1\}$. 
Now
\begin{equation*}
 \nabla_p \sigma(p)=
 \begin{bmatrix}
  e^{p_1-p_2} \\ -e^{p_3-p_2}-e^{p_1-p_2} \\ e^{p_3-p_2}+e^{p_3-p_4} \\ -e^{p_3-p_4}
 \end{bmatrix}
\in \spn\{e_3-e_4,e_2,e_1\} \Longleftrightarrow e^{p_3-p_2}+e^{p_3-p_4}=e^{p_3-p_4},
\end{equation*}
where the equivalence follows by comparing the third and fourth components of $\nabla_p \sigma(p)$, and the third and fourth components of any vector in $\spn\{e_3-e_4,e_2,e_1\}$.
Since $e^{p_3-p_2}>0$ for every $p_3,p_2\in\R$, it follows that $\nabla_p\sigma(p)\notin \spn\{e_3-e_4,e_2,e_1\}$, and thus $\mcal{J}(q)$ in \eqref{eq_two_compartment_model_coefficients_map_Jacobian} has full rank. Thus, \Cref{asmp_coefficient_map_Jacobian_full_rank} holds.
\end{proof}

\bibliographystyle{amsplain}
\bibliography{references}

\end{document}